\documentclass[11pt,a4paper]{article}
\usepackage[english]{babel}
\usepackage{amssymb}
\usepackage{amsmath}
\usepackage{indentfirst}
\usepackage[dvips]{graphicx}
\usepackage{epsfig}

\usepackage{color}

\usepackage{amsthm}
\usepackage{comment}

\newtheorem{cor}{Corollary}[section]
\newtheorem{prop}[cor]{Proposition}

\numberwithin{equation}{section}

\usepackage{amscd, amssymb, amsthm, amsmath}
\newtheorem{thm}{Theorem}

\newtheorem{lem}{Lemma}

\numberwithin{equation}{section}

\begin{document}
\title
{The second-order $L^2$-flow of inextensible elastic curves with hinged ends in the plane} 
\author{Chun-Chi Lin\footnote{Department of Mathematics, National Taiwan Normal University, 116 Taipei, Taiwan; chunlin@math.ntnu.edu.tw} \footnote{National Center for Theoretical Sciences, Taipei Office.}, 
Yang-Kai Lue\footnote{Departement of Mathematics, National Tsing-Hua University, 300 Hsinchu, Taiwan; luf961@yahoo.com.tw}, 
Hartmut R. Schwetlick\footnote{Department of Mathematical Sciences, University of Bath, UK; schwetlick@maths.bath.ac.uk}
} 
\date{January 13, 2014}
\maketitle

\begin{abstract}
In the paper published in Duke Math. J. 1993,
Y. Wen studied a second-order parabolic equation for inextensible 
elastic \emph{closed} curves in $\mathbb{R}^{2}$ toward inextensible elasticae. 
In this article, we extend Wen's result to the case of open inextensible planar curves with hinged ends. 
We obtain the long time existence of smooth solutions when the initial curves fulfill certain regular conditions.  
\end{abstract}


\baselineskip=15 pt plus .5pt minus .5pt

\section{Introduction}

One of the simplest mathematical settings in elastic mechanics is the so-called Euler-Kirchhoff theory 
(e.g., see \cite{Antman05} or \cite{LangSing96}). 
The other example is the mechanics of ribbons and M\"{o}bius strips (e.g., see \cite{SH07}), 
which recently has attracted more and more attention. 
These simple mechanical models providing challenging mathematical problems, 
e.g., in the calculus of variations and related dynamical theory. 
One of the simpler dynamical theory is gradient flows motivated from geometric variational functionals. 
They are often called geometric flows and can also be viewed as over-damped dynamics of mechanical objects, 
when the geometric objects are related to mechanics.

Geometric flows for open curves associated with energy functional of higher-order derivatives 
have appeared in various research topics, for examples, 
higher-order variational problems in differential geometry 
(e.g., \cite{BG86}, \cite{Willmore00}) 
and geometric control theory 
(e.g., \cite{Jurdjevic97}); 
interpolation problems of curves in computer-aided geometric design 
(e.g., \cite{BW98}, \cite{GJ82}, \cite{Mumford94}); 
mechanical modeling of polymers 
(e.g., mechanical modeling of DNA molecules \cite{OSC04} and filaments in biological cells \cite{OS10}). 
These functionals are often related to certain Sobolev norms of the first-order 
(e.g., stretching energy) and second-order derivatives (e.g., bending energy) of curves. 
The long-term existence of solutions of the so-called curve-straightening flow, 
using min-max method, has been established in the literature 
(e.g., \cite{LangSing85} on planar or spatial curves and \cite{LangSing87} 
on curves in a Riemannian manifold).  
Their min-max method works for the case of either open or closed curves. 
On the contrary, in the parabolic PDE approach, 
the long-time existence of smooth solutions has also been established for closed curves, 
e.g., in \cite{DKS02}, \cite{Koiso96}, \cite{Polden96}, \cite{Wen95}. 
There are however very few papers discussing the case of open curves in the PDE setting until recent years. 
To the best of our knowledge, they include: 
the $4$th-order flow of elastic curves with a positive stretching coefficient and clamped ends in \cite{Lin12}; 
the $4$th-order flow of anisotropic elastic curves with a positive stretching coefficient and hinged ends in \cite{DP13}; 
the $4$th-order flow of elastic curves with fixed length and hinged ends in \cite{DLP13}; 
the $4$th-order flow of elastic inextensible planar curves with hinged ends or infinite length in \cite{NO14}. 

Let $f:I=[0,L]\rightarrow \mathbb{R}^{2}$ be an open planar curve, 
$L>0$ represent the total length of $f$, 
and $s\in I=[0,L]$ be the arclength parameter of $f$. 
The curve is still said to be open even if $f(0)=f(L)$, since there is no regular conditions assumed at end points 
(e.g., we don't assume $\frac{df}{ds}(0) = \frac{df}{ds}(L)$).  
Denote by 
$T=\frac{df}{ds}$ the tangent vector of $f$ 
and by $\kappa=\frac{d^{2}f}{ds^{2}}$ the curvature vector of $f$. 
The bending energy of planar curves is defined by 
\begin{equation}
\mathcal{E}[f]
=\underset{I}{\int }\text{ }
\frac{1}{2}\left|\kappa\right|^{2}\text{ }ds
\label{eq:energy-signcurv}
\end{equation}
A regular planar curve 
$f_{0}:I\rightarrow \mathbb{R}^{2}$ is said to be inextensible
if its deformations are restricted to the class 
\begin{align}
&\mathcal{D}_{f_{0}}= 
\notag
\\
&\left\{
f\in C^{\infty }(I\times (-1,1),\mathbb{R}^{2}): 
f(s,0)=f_{0}, \frac{\partial f}{\partial s}(s,\varepsilon )=1, 
\text{ }\forall\text{ } s\in I, \forall \text{ }\varepsilon\in(-1,1) 
\right\}
\text{.}
\notag
\end{align} 
Suppose $f(0)=p_{-}$, $f(L)=p_{+}$ 
and 
$L > |p_{+}- p_{-}| $. 

The family of inextensible planar curves 
with fixed length $L$ and fixed end points 
$p_{-}, p_{+}$ 
can be equivalently described by the family of tangent vectors  
$T:I\times(-1,1) \rightarrow \mathbb{S}^{1}\left(1\right)\subset\mathbb{R}^{2}$ 
fulfilling the constraint 
\begin{equation}
\underset{I}{\int }\text{ }T\left( s, \varepsilon\right) \text{ }ds 
=p_{+}-p_{-}=:\triangle p 
\label{eq:0th-db}
\end{equation}
for all fixed $\varepsilon\in(-1,1)$. 
The admissible set 
\begin{equation}
\begin{array}{l}
\mathcal{A}_{L, \triangle p}
=\left\{
T\in C^{\infty }\left( [0,L],\mathbb{S}^{1} 
\left( 1\right)\right) : 
T \text{ satisfies (\ref{eq:0th-db})} 
\right\} 
\end{array} 
\label{eq:admissible-set} 
\end{equation} 
gives a family of inextensible planar curves 
with fixed length $L$ and fixed end points 
$p_{-}, p_{+}$. 
Therefore, instead of working with planar curves and their functional 
$\mathcal{E}$, 
we may equivalently consider tangent indicatrices and their functional 
$\mathcal{F}_{L,\triangle p}: 
\mathcal{A}_{L, \triangle p}\rightarrow \mathbb{R}$, 
\begin{equation}
\mathcal{F}_{L,\triangle p}\left[ T \right] 
=\underset{I}{\int }\text{ } 
\frac{1}{2} |\partial_{s} T|^{2}\text{ } ds
+\vec\lambda \cdot \left[\underset{I}{\int }\text{ }T\text{ }ds-\triangle p \right]
\label{eq:energy-F[T]}
\end{equation}
where $\vec\lambda=(\lambda_{1},\lambda_{2})$ 
is the $\mathbb{R}^{2}$-valued Lagrange multiplier.

In this article, we investigate the parabolic equation of tangent vector $T$ 
\begin{equation}
\partial_{t}T= \nabla_{s}^{2} T -\langle\vec\lambda, T^{\perp}\rangle T^{\perp}  
\label{eq:flow-T}
\end{equation}
with hinged boundary condition 
\begin{equation} 
\kappa_{|\partial I}=0
\label{eq:bc-hinged}
\end{equation}
and certain regular initial data $T_0=T_{|t=0}$ (see Theorem \ref{thm:main} for details). 
The equation (\ref{eq:flow-T}) is the $L^2$-flow of 
$\mathcal{F}_{L,\triangle p}$ in the class $\mathcal{A}_{L,\triangle p}$. 
Our result extends the work of Y. Wen \cite{Wen93} from the case of closed planar curves 
to that of open planar curves with hinged boundary conditions and sufficiently smooth initial data. 
Note that the $L^2$-flow of $\mathcal{F}_{L,\triangle p}$ 
in this article is a second-order nonlinear parabolic partial differential equation, 
whose leading terms are however linear. 
The equation is the same as the one in \cite{Wen93}, 
but is different from the fourth-order quasilinear parabolic equations induced from 
the $L^2$-gradient flow of elastic curves discussed in \cite{DKS02}, \cite{Polden96}, \cite{Wen95}. 
The second-order parabolic equation demonstrates nice geometric properties during evolution: e.g., 
convexity-preserving, non-increasing of the number of inflection points 
(e.g., see \cite{Ang91}, \cite{Wen93}). 
These geometric properties are analogous to those of curve-shortening flow 
(e.g., see \cite{GH86}) and 
might be useful in applications. 
In contrast, Linn\'{e}r gave an example showing the failure of convexity-preserving 
during the evolution of the so-called curve-straightening flow of planar curves in 
\cite{Linner89}. 
To the best of the authors' knowledge, 
it is not clear yet whether or not these geometric properties hold 
for the fourth-order $L^2$-flow of elastic curves studied in, 
e.g., \cite{DKS02}, \cite{Lin12}, \cite{Polden96}, \cite{Wen95}.


The long time existence of smooth solutions for the $L^2$-flow in this article 
is derived by applying the Gagliardo-Nirenberg type inequalities 
in the estimates of $L^{2}$ norms of high-order derives of curvature. 
The main difficulty in extending the work of \cite{Wen93} 
to the case of open curves comes from the extra terms associated to the boundary conditions. 
These extra terms can't be taken care by 
the Gagliardo-Nirenberg type inequalities in the $L^{2}$-estimates. 
In \cite{Lin12}, we found that good boundary conditions give certain algebraic relationship to derive the ``higher-order energy identities", 
which allow us to obtain the uniform bounds of higher-order derivatives of curvature. 
Namely, as replacing the estimates of $\|\nabla_s^m\kappa\|_{L^2}$ by 
$\|\nabla_t^m f\|_{L^2}$, where $f$ is the position vector of curves, 
we avoid the estimates of boundary terms in the case of clamped ends. 
In this article, we however found that the approach through estimating terms 
$\|\nabla_s^m\kappa\|_{L^2}$ still works 
in the case of hinged boundary conditions 
as one carefully utilize 
the hinged boundary conditions for the parabolic equation (\ref{eq:flow-T}), 
see Lemma \ref{lem:nabla_s^2m(kappa)=0_at_bdry}. 
Besides, the flow with hinged ends is a very natural problem to study, 
both in mechanical terms and in a p.d.e. setting. 
On the other hand, in contrast to \cite{Lin12}, 
we are not able to obtain the long time existence of smooth solutions 
for the second-order $L^2$-flow with clamped boundary conditions in this article. 
In fact, if one chooses the approach in \cite{Lin12} to study the flow for inextensible planar elasticae with clamped ends, 
the term $\|\nabla_t^m f\|_{L^2}$ contains time-derivatives of Lagrange multipliers, 
which provide additional difficulties in applying the Gagliardo-Nirenberg type inequalities. 
We thus leave this case to the future work.

Below is the main result of this article. 

\begin{thm}
\label{thm:main}
For any given initial $C^\infty$-smooth and open inextensible planar curve 
$f_{0}:[0,L]\rightarrow\mathbb R^2$ with
$\kappa_0^{(\ell)}(0)=\kappa_0^{(\ell)}(L)=0$, 
$\forall\text{ }\ell\in\mathbb N\cup\{0\}$,  
there exists a $C^\infty$-smooth global solution of the $L^{2}$-flow 
(\ref{eq:flow-T}) with the hinged boundary conditions (\ref{eq:bc-hinged}) 
for the bending energy of curves. 
Moreover, the inextensible curves subconverge to an inextensible elastica, 
i.e., an equilibrium configuration of the functional 
$\mathcal{E}$ in the class $\mathcal{D}_{f_0}$ 
with the hinged boundary condition 
(\ref{eq:bc-hinged}) and fixed end points $f_0(0)$, $f_0(L)$.
\end{thm}

The rest of this article is organized as the following. 
In Section 2, we derive the Euler-Lagrange equation and set up 
the second-order parabolic equation of the $L^2$-flow (\ref{eq:flow-T}). 
The proof of our main result, Theorem \ref{thm:main}, is contained in Section 3, 
where the part of short time existence of $C^\infty$-smooth solutions is a bit lengthy 
because it is presented by an elementary argument. 

\section{The $L^{2}$-flow equation}

The first variation of $\mathcal{F}_{L,\triangle p}\left[ T \right]$ in the class $\mathcal{A}_{L,\triangle p}$ gives 
\begin{equation*}
\delta \mathcal{F}_{L,\triangle p}\left[ T \right] 
= \langle \kappa, (\partial_{\varepsilon}T)_{|\varepsilon=0} \rangle_{|\partial I}
- 
\underset{I}{\int }
\langle \partial_{s}\kappa -\vec\lambda, 
(\partial_{\varepsilon}T)_{|\varepsilon=0} \rangle \text{ } ds 
\text{,}
\end{equation*}
where $\kappa=\partial_{s} T$ is the curvature vector of a planar curve $f$. 
Since $|T|\equiv1$ implies $\partial_{\varepsilon} T \perp T$, 
we rewrite this equation as 
\begin{equation}
\delta \mathcal{F}_{L,\triangle p}\left[ T \right] 
= 
- \underset{I}{\int }
\langle \nabla_{s}\kappa -\langle\vec\lambda, T^{\perp}\rangle T^{\perp}, 
(\partial_{\varepsilon}T)_{|\varepsilon=0} \rangle \text{ } ds 
\text{,}
\label{eq:1st-var-2}
\end{equation}
where $\nabla_{s} g:=\langle \partial_{s} g, T^{\perp}\rangle T^{\perp}$ 
and $T^{\perp}$ is a unit normal vector of the curve $f$ 
derived from rotating its unit tangent vector $T$ counterclockwise. 
There is no boundary term in (\ref{eq:1st-var-2}), because of the hinged boundary conditions. 
By assuming that $T$ is a critical point of $\mathcal{F}_{L,\triangle p}$ 
in $\mathcal{A}_{L, \triangle p}$, 
we obtain the Euler-Lagrange equation of $T$, 
\begin{equation}
\nabla_{s}^{2} T -\langle\vec\lambda, T^{\perp}\rangle T^{\perp} =0 
\text{.}
\label{eq:EL-T}
\end{equation}


From (\ref{eq:1st-var-2}), (\ref{eq:flow-T}) and (\ref{eq:bc-hinged}), 
one obtains the energy identity 
\begin{equation}
\frac{d}{d t} \mathcal{F}_{L,\triangle p}\left[ T \right] 
=- \underset{I}{\int }| \partial_{t}T |^{2} \text{ } ds 
\text{.}
\label{eq:energy_ID-1}
\end{equation}
It implies the non-increasing property of $\mathcal{F}_{L,\triangle p}[T_{t}]$ 
\begin{equation}
\mathcal{F}_{L,\triangle p}[T_{t}]\leq \mathcal{F}_{L,\triangle p}[T_{0}] 
\text{, } \forall\text{ } t \in (0, t_0) 
\label{eq:curv-esti-0th}
\end{equation} 
as the smooth solutions of the $L^2$-flow (\ref{eq:flow-T}) exist 
$\forall \text{ } t\in (0,t_0)$. 
Note that, in (\ref{eq:flow-T}), one may write 
\[
\langle \vec\lambda, T^{\perp} \rangle \text{ } T^{\perp}
= \vec\lambda \cdot [(T^{\perp})^{t} T^{\perp}] 
\text{,}
\]  
where $[(T^{\perp})^{t} T^{\perp}] \in\mathbb{M}_{2\times 2}$ 
(the set of $2\times 2$ matrices). 
Therefore, from the constraint  
\begin{equation*}
0=\frac{d}{dt}\underset{I}{\int}\text{ }T\text{ }ds
= \underset{I}{\int }\text{ }\partial_{t}T\text{ }ds 
\end{equation*}
and (\ref{eq:flow-T}), 
one has 
\begin{equation*}
\underset{I}{\int} \text{ } \nabla_{s}^{2} T \text{ }ds 
=\vec\lambda\cdot \underset{I}{\int}\text{ }[(T^{\perp})^{t} T^{\perp}] \text{ }ds 
=\vec\lambda\cdot A_T 
\end{equation*}
with  
\begin{equation}
A_T
:= \underset{I}{\int}\text{ }[(T^{\perp})^{t} T^{\perp}] \text{ }ds 
\text{.}
\label{eq:A}
\end{equation}
If $\det A_T\ne 0$, the vector-valued Lagrange multiplier $\vec\lambda$ can be written as 
\begin{equation}
\vec\lambda 
= 
\left(
\underset{I}{\int}\text{ }\nabla_{s}^{2} T \text{ }ds 
\right) 
\cdot 
A_T^{-1}
\text{.}
\label{eq:lambda-T}
\end{equation}
By applying the hinged boundary condition (\ref{eq:bc-hinged}) 
and integration by parts, 
the vector-valued Lagrange multiplier $\vec\lambda$ can be rewritten as  
\begin{equation}
\vec\lambda 
= 
\left(\underset{I}{\int}\text{ } |\kappa|^{2} \text{ }T \text{ }ds \right) 
\cdot 
A_T^{-1}
\text{.}
\label{eq:lambda-T-hingedBC}
\end{equation} 
From the evolution equation of the tangent vector in (\ref{eq:flow-T}) 
and the property $\nabla_s T^{\perp}=0$, 
we obtain 
\[
\nabla_{t} \kappa 
=[\partial_t \kappa]^\perp 
=[\partial_t \partial_s T]^\perp
=[\partial_s \partial_t T]^\perp
=[\partial_s (\nabla_s \kappa -\langle \vec\lambda, T^\perp \rangle T^\perp)]^\perp 
\text{,}
\]
which gives  
\begin{equation}
\nabla_{t}\kappa= \nabla_{s}^{2} \kappa + \langle\vec\lambda, T\rangle \kappa  
\text{.}
\label{eq:flow-kappa}
\end{equation}

By introducing the tangent indicatrix $T_\varphi$,
\begin{equation}
T=T_\varphi=\left( \cos \varphi ,\sin \varphi \right)
\text{,}
\label{eq:T=(cos,sin)} 
\end{equation}
where $\varphi:I\times[0,t_1)\rightarrow\mathbb R$, 
we can transfer the evolution equation (\ref{eq:flow-T}) into a scalar equation,  
\begin{equation*}
(\partial_{t}\varphi - \partial_{s}^{2}\varphi 
-\lambda_{1}\sin \varphi +\lambda_{2}\cos \varphi ) \cdot T^{\perp}
=0 
\text{,} 
\end{equation*} 
or equivalently 
\begin{equation}
\partial_{t}\varphi 
=\partial_{s}^{2}\varphi 
- \langle \vec\lambda, T^{\perp} \rangle 
\text{.}  
\label{eq:flow-phi}
\end{equation} 
Similarly, 
the evolution equation for curvature vector $\kappa$ in (\ref{eq:flow-kappa}) 
can be written in terms of the signed curvature $k=\partial_s \varphi$ as 
\begin{equation}
\partial _{t}k=\partial_{s}^{2} k + \langle \vec\lambda, T \rangle k 
\text{.}  
\label{eq:flow-k}
\end{equation} 
Using the expression of tangent indicatrix $T$ in (\ref{eq:T=(cos,sin)}), 
(\ref{eq:A}) yields 
\begin{equation}
A_T 
=\left(
\begin{array}{cc}
\underset{I}\int \text{ }\sin^2\text{}\varphi\text{ }ds & -\underset{I}\int \text{ }\sin\text{}\varphi\cos\text{}\varphi\text{ }ds 
\\
-\underset{I}\int \text{ }\sin\text{}\varphi\cos\text{}\varphi\text{ }ds & \underset{I}\int \text{ }\cos^2\text{}\varphi\text{ }ds
\end{array}
\right) 
\text{,}
\label{eq:matrix_A}
\end{equation}
and hence 
\begin{equation}
\det A_T=
\left( \underset{I}{\int}\text{ }\cos^{2}\varphi \text{ }ds\right)
\cdot \left( \underset{I}{\int}\text{ }\sin^{2}\varphi \text{ }ds\right) 
-\left( 
\underset{I}{\int}\text{ }\cos \varphi \sin \varphi \text{ }ds
\right)^{2}
\text{,}
\label{eq:lambda-phi-Det}
\end{equation}
which is non-negative by Cauchy-Schwartz inequality. 
Moreover, 
from (\ref{eq:lambda-T-hingedBC}), 
the Lagrange multipliers in the case of hinged boundary condition 
can be expressed as 
\begin{equation}
\left\{ 
\begin{array}{l}
\lambda_1
=
\left(\det A_T \right)^{-1}
\cdot 
\left[
\left( \underset{I}{\int}\text{ }(\partial _{s}\varphi)^2 \cos \varphi 
\text{ }
ds\right) 
\cdot \left( \underset{I}{\int}\text{ }\cos^{2}\varphi \text{ }ds\right)
\right.
\\
\left.
\text{ \  \  \  \  \  \  \  \  \  \  \  \  \  \  \  \  \  \  \  \  \  \  \  \  \  \  \  \  \  \  \  \ }
+ \left( \underset{I}{\int }\text{ }
(\partial _{s}\varphi)^2 \sin \varphi \text{ }
ds\right) \cdot \left( \underset{I}{\int}\text{ }\sin \varphi \cos \varphi 
\text{ }
ds\right)
\right] 
\text{,} 
\\ \\ 
\lambda _2
=
\left(\det A_T \right)^{-1}
\cdot
\left[
\left( \underset{I}{\int}\text{ }(\partial _{s}\varphi)^2 \cos \varphi 
\text{ }ds\right) \cdot \left( \underset{I}{\int}\text{ }\sin \varphi \cos \varphi 
\text{ }ds\right) 
\right.
\\ 
\left.
\text{ \  \  \  \  \  \  \  \  \  \  \  \  \  \  \  \  \  \  \  \  \  \  \  \  \  \  \  \  \  \  \  \ }
+ \left( \underset{I}{\int}\text{ } (\partial _{s}\varphi)^2 \sin \varphi 
\text{ }ds\right) 
\cdot \left( \underset{I}{\int}\text{ }\sin^{2}\varphi 
\text{ }ds \right) 
\right]
\text{.}
\end{array}
\right. 
\label{eq:lambda-phi-hingedBC}
\end{equation}

\section{The existence of global smooth solutions} 

In this section, we prove long time existence of smooth solutions. 
The key step is to show that 
$\left\Vert \partial_s^m\kappa \right\Vert_{L^2}$ 
remains uniformly bounded during the geometric flow (\ref{eq:flow-T}) 
for any $m\in\mathbb N$.

\subsection{Some preliminary estimates} 

The following lemma gives a crucial estimate on the Lagrange multipliers.

\begin{lem}
\label{lem:lambda} 
Let $T\in \mathcal{A}_{L, \triangle p}$, 
where $\triangle L= L-|\triangle p|>0$ and 
$\left\| \partial_s T \right\| _{L^{2}}\leq M$. 
Suppose $T$ satisfies the hinged boundary condition, i.e., $\partial_s T=0$ on the boundary $\partial I$. 
Then, there exist positive numbers 
$C_1=C_1\left( \triangle L, L, M\right)$ and $C_2=C_2\left( \triangle L, L, M\right)$ such that 
$(i)$ $C_1\leq \det A_T$; $(ii)$ $|\vec\lambda|\leq C_2$. 
\end{lem}

\begin{proof}
Observe that 
\begin{equation*}
\begin{array}{l}
2 \cdot \det A_T 
\\ 
=\left(\underset{ I}{\int}\text{ }\cos^{2}\varphi (s)\text{ }ds \right)
\cdot 
\left(\underset{ I}{\int}\text{ }\sin^{2}\varphi (\sigma )\text{ }d\sigma \right)
\\
\text{ \ }
-2\left(
\underset{I}{\int}\text{ }\cos \varphi (s) \sin \varphi (s)\text{ }ds
\right)
\cdot 
\left(
\underset{I}{\int}\text{ } \cos \varphi (\sigma )\sin \varphi (\sigma )
\text{ }d\sigma 
\right) 
\\ 
\text{ \ }
+
\left(
\underset{I}{\int}\text{ }\cos^{2}\varphi (\sigma )
\text{ }d\sigma 
\right)
\cdot 
\left(
\underset{I}{\int}\text{ }\sin^{2}\varphi (s)\text{ }ds
\right)
 \\ \\
=\underset{I\times I}{\iint }
\left[
\cos^{2}\varphi (s)
\sin^{2}\varphi (\sigma)
-2\cos \varphi (s)\sin \varphi (s)
\cos \varphi (\sigma)\sin \varphi (\sigma)
\right.
\\ 
\text{ \ }
\left. 
\text{ \  } \text{ \  } \text{ \  } \text{ \  }
+\cos^{2}\varphi (\sigma )\sin^{2}\varphi (s)
\right]
\text{ }d\sigma ds 
\\ \\
=\underset{I\times I}{\iint}[\cos \varphi (s)\sin \varphi
(\sigma )-\cos \varphi (\sigma )\sin \varphi (s)]^{2}\text{ }d\sigma ds 
\\ \\
=\underset{I\times I}{\iint}\sin ^{2}
\left[ \varphi (\sigma)-\varphi (s) \right]
\text{ }d\sigma ds
\geq 0
\text{.}
\end{array}
\end{equation*}
Since $\varphi\in W^{1,2}(I)$, 
by Morrey's inequality 
(cf. \cite{Evans92}),
there exists a 
constant $C_0>0$, independent of $\varphi$ and $|I|$,  
such that $\varphi \in C^{0,1/2}(I)$ and 
\begin{equation}
|\varphi \left( x\right) -\varphi \left( y\right) |
\leq C_0 \cdot 
\left(
\underset{I} {\int}\text{ }|\partial _{s}\varphi |^{2}\text{ }ds
\right)^{1/2}
\cdot |x-y|^{1/2}
\label{eq:m-s}
\end{equation}
for any $x$, $y\in I$. 
By rotations of the coordinate of $\mathbb R^2$, where the planar curves sit in, 
we may assume  
\begin{equation}
\underset{I}{\int }
\left(
\cos \varphi \left( s\right) ,\sin \varphi \left(s\right) 
\right)
\text{}ds
=(|\triangle p|,0)
\text{,}
\label{eq:delta-p}
\end{equation}
which implies 
\begin{equation}
\underset{I}{\int }
\left[1-\cos \varphi \left( s\right) \right]
\text{ }ds
=\triangle L>0 
\text{.}
\label{eq:delta_L}
\end{equation}
Thus, from applying (\ref{eq:delta-p}), (\ref{eq:delta_L}) 
and the mean value theorem, 
there exist $s_{-}, s_{+}\in I$ such that 
$\sin\varphi(s_{-})=0$
and 
$1-\cos\varphi(s_{+})=\frac{\triangle L}{L}\in(0,1)$.  
Let 
\begin{equation}
\left\{ 
\begin{array}{l}
\varphi_{-}:=\varphi(s_{-})\in\{n\pi: n\in\mathbb Z\}
\text{,}
\\ 
\varphi_{+}:=\varphi(s_{+})
=\arccos(1-\triangle L/L)\notin \{n\pi: n\in\mathbb Z\}
\text{,}
\end{array}
\right.  
\label{eq:phi_{+,-}}
\end{equation}
and  
\begin{equation}
|\triangle \varphi|:=|\varphi_{+}-\varphi_{-}|\ge |\arccos(1-\triangle L/L)| 
\in (0, \pi/2) 
\text{.}
\label{eq:delta-varphi}
\end{equation}

Choose 
\[
\delta_0=\delta_0 (\triangle L, L, M)
:=\left(\frac{|\triangle\varphi|}{3C_0 M}\right)^2
\]
and 
\[
J_{s}(r):=\left\{\sigma\in \mathbb{R}: |\sigma-s |\le r \right\}
\cap I
\text{.}
\] 
Note that  
$|J_{s}(r)|\ge r$, as $0<r\le |I|$. 
By applying (\ref{eq:m-s}), 
we have 
\begin{equation}
\begin{array}{l}
|\varphi (s)-\varphi (s_\pm)| 
\leq 
\frac{|\triangle \varphi|}{3}
\text{ , } \forall \text{ }
s\in J_{s_\pm} ( \delta_0 ) 
\text{ , } \forall \text{ } \pm\in\{-,+\}
\text{.}
\end{array}
\label{eq:tri-ineq_+-}
\end{equation} 
From (\ref{eq:tri-ineq_+-}), we conclude 
\begin{equation}
|\varphi (\sigma_1)-\varphi (\sigma_2 )| 
\ge \frac{|\triangle \varphi|}{3} 
\text{, }\forall \text{ } 
\sigma_1\in J_{s_-}(\delta_0)
\text{, }\forall \text{ } 
\sigma_2 \in J_{s_+}(\delta_0) 
\text{.}
\label{eq:lower_bdd_phi}
\end{equation} 
Therefore, 
\begin{eqnarray*}
\det A_T
&\ge&
\frac{1}{2}\underset{J_{s_-}(\delta_0)\times J_{s_+}(\delta_0)}
{\int}\text{ }
[\sin (\varphi (\sigma_1)-\varphi (\sigma_2))]^2 
\text{ }d\sigma_1 d\sigma_2 
\\ 
&\ge& \frac{1}{2}\left(\sin \frac{|\triangle \varphi|}{3}\right)^2\left| J_{s_-}(\delta_0)\times J_{s_+}(\delta_0)\right| 
\\
&\ge & \frac{1}{2}\left(\delta_0 \cdot \sin \frac{|\triangle \varphi|}{3}\right)^2 
=:C_1=C_1 (\triangle L, L, M) > 0 
\text{.}
\end{eqnarray*}
This completes the proof of $(i)$. 
The conclusion $(ii)$ is a consequence of applying (\ref{eq:lambda-T-hingedBC}) 
and result $(i)$. 
\end{proof}

To simplify the presentation, we define some notation.   
For normal vector fields $\phi _{1},\cdot \cdot \cdot ,\phi _{\ell}$ 
along a curve $f$,
we denote by $\phi _{1}\ast \ast \ast \phi _{\ell}$ a term of the type
\begin{equation*}
\phi _{1}\ast \ast \ast \phi _{\ell}=\left\{
\begin{array}{l}
\langle \phi _{i_{1}},\phi _{i_{2}}\rangle \cdot \cdot \cdot \langle \phi
_{i_{\ell-1}},\phi _{i_{\ell}}\rangle \text{ \ \ \ \ \ \ \ , for }\ell\text{\ even,} 
\\
\langle \phi _{i_{1}},\phi _{i_{2}}\rangle \cdot \cdot \cdot \langle \phi
_{i_{\ell-2}},\phi _{i_{\ell-1}}\rangle \cdot \phi _{i_{\ell}}\text{, for }
\ell\text{\
odd,}
\end{array}
\right.
\end{equation*}
where $(i_{1},...,i_{\ell})$ is any permutation of $(1,..., \ell)$. 
Slightly more general, we allow some of the $\phi _{i}$ to be
functions, in which case the $\ast $-product reduces to a multiplication. 
Denote by $P_{b}^{a,c}(\phi)$ any linear combination of terms of the type 
\[
\nabla_{s}^{i_1} \phi \ast \dots \ast \nabla_{s}^{i_{b}} \phi,  
\text{ }\mbox{}i_1 + \dots +i_{b}= a \mbox{ with } \max i_{j} \leq c 
\text{,}
\] 
where all coefficients are bounded from above and below 
by some universal constants, depending only on $a$ and $b$. 
Moreover, let 
\begin{equation}
\sum_{\substack{[\![a,b]\!] \leq [\![A,B]\!]\\c\leq C}} P^{a,c}_{b} (\kappa) 
: = \sum_{a=0}^{A} \sum_{b=1}^{2A+B-2a} \sum_{c=0}^{C}  P^{a,c}_{b}(\kappa) 
\label{eq:sum_P^a_b}
\text{,}
\end{equation} 
where $[\![a,b]\!]:=2a+b$.

\begin{lem}
\label{lem:Lin_Lemma8-like} 
Suppose 
$T: I\times[0,t_1) \rightarrow \mathbb{R}^{2}\cap\mathbb{S}^1$ 
is a smooth solution of 
(\ref{eq:flow-T}). 
Let 
$\phi_{\ell}:=\nabla_{s}^{\ell} \kappa$. 
Denote by $\psi:I\times[0,t_1)\rightarrow \mathbb{R}^{2}$ 
a smooth normal vector field along the planar curve, 
i.e., $\psi(s,t)\perp T(s,t)$, for all $(s,t)\in I\times [0,t_1)$. 
Then, for any integer $\ell\ge2$ and $k, m\in\mathbb N$, 
we have the following formulae, 
\begin{align}
\nabla_{t}\nabla_{s} \psi= \nabla_{s}\nabla_{t} \psi 
\text{,}
\label{eq:Lin_Lemma8-like-1}
\end{align}
\begin{align}
\partial_s \psi=\nabla_{s} \psi -\langle \psi, \kappa\rangle  T 
\text{,}
\label{eq:Lin_Lemma8-like-2}
\end{align}
\begin{align}
\partial_{s}^{\ell}\kappa - \nabla_{s}^{\ell} \kappa 
=
\left[\sum_{\substack{[\![a,b]\!] \leq [\![\ell-1,2]\!]\\ c\leq \ell-1}} P^{a,c}_{b} (\kappa) \right]  T 
+\sum_{\substack{[\![a,b]\!] \leq [\![\ell-2,3]\!]\\ c\leq \ell-2}} P^{a,c}_{b} (\kappa) 
\text{,}
\label{eq:Lin_Lemma8-like-4}
\end{align}
where the coefficient of $T$ in (\ref{eq:Lin_Lemma8-like-4}), 
i.e., the one denoted by $[\cdots]$, is a sum of terms like  
$\langle \phi _{i_{1}},\phi _{i_{2}}\rangle \cdot \cdot \cdot 
\langle \phi_{i_{j-1}},\phi _{i_{j}}\rangle $; 
while the last term in (\ref{eq:Lin_Lemma8-like-4}) is a sum of terms of the form  
$\langle \phi _{i_{1}},\phi _{i_{2}}\rangle \cdot \cdot \cdot 
\langle \phi_{i_{k-2}},\phi _{i_{k-1}}\rangle \cdot \phi _{i_{k}}$. 
\end{lem}

\begin{proof}

$(i)$ 
We obtain (\ref{eq:Lin_Lemma8-like-1}) by applying the property: 
if $\psi=\varphi \cdot T^\perp$, then
$\nabla_s \psi=\partial_s \varphi \cdot T^\perp$ and 
$\nabla_t \psi=\partial_t \varphi \cdot T^\perp$. 

$(ii)$ 
We derive (\ref{eq:Lin_Lemma8-like-2}) from  
\[
\partial_s \psi=\nabla_{s} \psi +\langle \partial_s\psi, T\rangle  T
=\nabla_{s} \psi -\langle \psi, \kappa\rangle  T 
\text{.}
\]

$(iii)$ 
The proof of (\ref{eq:Lin_Lemma8-like-4}) is an induction argument. 
From (\ref{eq:Lin_Lemma8-like-2}), we have 
$\partial_{s} \kappa = \nabla_s \kappa - |\kappa|^2 T$. 
Thus, by applying (\ref{eq:Lin_Lemma8-like-2}) again, we have 
$\partial_{s}^2 \kappa = \partial_s (\nabla_s \kappa - |\kappa|^2 T)
=\nabla_s^2 \kappa - 3\langle \nabla_s\kappa, \kappa\rangle T -  |\kappa|^2 \cdot \kappa $, 
which proves the case of $\ell=2$ in (\ref{eq:Lin_Lemma8-like-4}). 
For $\ell\ge3$, we apply (\ref{eq:Lin_Lemma8-like-2}) and prove (\ref{eq:Lin_Lemma8-like-4}) easily by induction. 
Thus, we leave it to the reader to verify.

\end{proof}

We recall the Gagliardo-Nirenberg type interpolation
inequalities from Proposition 2.5 in \cite{DKS02} 
in terms of the modified notation $P^{a,c}_{b}(\kappa)$. 
\begin{lem}
\label{lem:DKS2.5}For any term 
$P_{b }^{a, c} (\kappa) $ 
with $b\geq 2$, 
which contains only derivatives of $\kappa$ with order at most $m-1$, 
we have 
\begin{equation}
\begin{array}[t]{l}
\underset{I}{\int}\text{ }\left| P_{b }^{a, c}\left( \kappa\right) \right| 
\text{ }ds
\leq 
C \cdot \text{ }\mathcal{L}\left[ f\right] ^{1-a -b }\left\|
\kappa\right\| _{2}^{b -\gamma}\left\| \kappa\right\| _{m, 2}^{\gamma }
\text{ ,}
\end{array}
\label{eq:DKS2.15}
\end{equation}
where 
$\gamma =\left(a +\frac{b}{2} -1\right) /m$, 
$C=C\left( n,m,a,b \right) $, 
$\mathcal{L}\left[ f\right] $\ the length of $f$, and 
\begin{equation*}
\left\| \kappa\right\| _{m,p}:=\overset{m}{\underset{i=0}{\sum }}
\left\| \nabla_{s}^{i}\kappa\right\| _{p}\text{, }\left\| \nabla_{s}^{i}
\kappa\right\|_{p}:=
\mathcal{L}\left[ f\right] ^{i+1-1/p}
\left(\underset{I}{\int}\text{ }|\nabla_{s}^{i}\kappa|^{p}\text{ }ds\right)^{1/p}
\text{.}
\end{equation*}
Moreover, if $a +\frac{b}{2} <2m+1$, then $\gamma<2$ 
and we have for any $\varepsilon >0$ 
\begin{equation}
\begin{array}[t]{l}
\underset{I}{\int}\text{}\left| P_{b }^{a, c}\left(\kappa\right)\right|\text{}ds
\leq \varepsilon \text{}
\underset{I}{\int}\text{}|\nabla_{s}^{m}\kappa|^{2}\text{}ds
+C\text{}\varepsilon ^{\frac{-\gamma }{2-\gamma }}
\left(\underset{I}{\int}\text{}|\kappa|^{2}\text{}ds
\right)^{\frac{b -\gamma}{2-\gamma }}
+ C\text{}\left(\underset{I}{\int}\text{ }|\kappa|^{2}\text{}ds\right)^{a +b -1}
\text{.}
\end{array}
\label{eq:DKS2.16}
\end{equation}
\end{lem}

\subsection{Short time existence}

In this part, we give a proof of the short time existence of smooth solutions for (\ref{eq:grad.flow.H}) below, 
i.e., (\ref{eq:flow-phi}) with hinged boundary condition. 
We follow the proof of short-time existence of classical solutions for the semilinear parabolic equations in \cite{Cannon84}, 
where a contraction map on a proper functional space plays the key role.

\begin{thm}
[Short time existence for hinged boundary condition]
\label{thm:STE_for_N}
For a given $\varphi_0 \in C^\infty ([0,L])$ with
$\varphi_0^{(\ell)}(0)=\varphi_0^{(\ell)}(L)=0$, 
$\forall\text{ }\ell\in\mathbb N$,
there is a positive time $t_0$ 
and a unique smooth function $\varphi(s,t)$ satisfying
\begin{align}
\begin{cases}
\partial_t\varphi=\partial^2_s\varphi+\lambda_1\sin\varphi-\lambda_2\cos\varphi
~&\mbox{in}~ (0,L)\times (0,t_0) 
\text{,}
\\
\varphi (s,0)=\varphi_0(s) ~&\mbox{}~\forall \text{ }  0\leq s\leq L 
\text{,}
\\
\partial_s\varphi(0,t)=\partial_s\varphi(L,t)=0 &\mbox{}~ \forall \text{ } 0\leq t\leq t_0 
\text{.}
\end{cases}
\label{eq:grad.flow.H}
\end{align}
\end{thm}

Instead of proving existence of smooth solutions of (\ref{eq:grad.flow.H}) directly, 
we first work with the equation for a function,  
$\varphi: \mathbb{R}\times(0,t_0)\rightarrow\mathbb{R}$, 
on the entire domain, 
\begin{align}
\begin{cases}
\partial_t\varphi=\partial^2_s\varphi+\lambda_1\sin\varphi-\lambda_2\cos\varphi 
~&\mbox{in}~ \mathbb R\times (0,t_0) 
\text{,}
\\
\varphi (s,0)=\varphi_0(s) ~&\mbox{}~ \forall\text{ } s\in\mathbb R 
\text{,}
\end{cases}
\label{eq:grad.flow.G}
\end{align}
where 
$\lambda_1=\lambda_1(t)$ and $\lambda_2=\lambda_2(t)$ 
are determined by the integrals over the finite domain $[0,L]$ 
as shown in (\ref{eq:lambda-phi-Det}), (\ref{eq:lambda-phi-hingedBC}).

\begin{prop}
\label{prop:Short-T.E.} 
Let $\varphi_0:\mathbb R\rightarrow\mathbb R$ be a smooth function satisfying 
\begin{equation}
\underset{\mathbb R}\sup\left\{\left|\varphi_0^{(\ell)}\right|\right\}\le K_\ell 
\label{eq:varphi_0_C^infty}
\end{equation} 
for some bounded sequence of numbers $\{K_\ell: \ell\in\mathbb N\cup\{0\}\}$. 
Then there exists $t_0>0$ 
and a unique $C^\infty$-smooth solution of (\ref{eq:grad.flow.G}) such that 
\[
\underset{t\rightarrow +0}\lim\text{ }\varphi^{(\ell)} (s,t)
=\varphi_0^{(\ell)} (s)
\]
$\forall \text{ }\ell\in\mathbb N\cup\{0\}$, $\forall \text{ }s\in [0,L]$, 
$\forall \text{ }t\in(0,t_0)$. 
\end{prop}

For the proof of Proposition \ref{prop:Short-T.E.}, we cite the following lemma from \cite{Cannon84}. 

\begin{lem}[Lemma 19.2.1 of \cite{Cannon84}]
\label{lem:19.2.1_Cannon84}
For bounded continuous $f$ in $-\infty<x<\infty$, $t\ge 0$, which is uniformly H\"{o}lder with exponent 
$\alpha$, $0<\alpha<1$, with respective to $x$, the potential 
\[
z(x,t)=\int_0^t\int_{-\infty}^\infty \text{ } K(x-\xi,t-\tau) f(\xi,\tau) \text{ } d\xi d\tau
\] 
possesses the following properties:
\begin{itemize}
\item[1)] $z$, $z_x$, $z_t$, and $z_{xx}$ are continuous; 
\item[2)] $z_t=z_{xx}+f(x,t)$, $-\infty<x<\infty$, $0< t$; 
\item[3)] $|z(x,t)|\le t \cdot \|f\|_t$, $-\infty<x<\infty$, $0\le t$, 
where $\|f\|_t:=\underset{x\in\mathbb R;\tau\in[0,t]}\sup |f(x,\tau)|$; 
\item[4)] $|z_x (x,t)|\le 2\pi^{-1/2} \|f\|_t \cdot t^{1/2}$, $-\infty<x<\infty$, $0\le t$; 
\item[5)] $|z_{xx}(x,t)|\le C |f|_\alpha \cdot t^{\alpha/2}$, $-\infty<x<\infty$, $0\le t$, 
where $C$ is a positive number and 
\[
|f|_\alpha:=\underset{x\in\mathbb R; 0\le t; \delta>0}\sup 
\left\{ \frac{|f(x+\delta,t)-f(x,t)|}{\delta^{\alpha}}\right\}
\text{;}
\] 
\item[6)] $|z_t (x,t)|\le \|f\|_t+C |f|_\alpha \cdot t^{\alpha/2}$, $-\infty<x<\infty$, $0\le t$. 
\end{itemize}
\end{lem}

We introduce some notation for the proof of Proposition \ref{prop:Short-T.E.}. 
For any $a>0$, 
let 
$D_a:=\{(s,t): s\in\mathbb{R}, t\in(0, a)\}$, 
$D_a^L:=\{(s,t): s\in[0,L], t\in(0, a)\}$, 
and 
\begin{align}
\mathfrak{B}_{a}:=
\{\psi: D_{a}\rightarrow\mathbb{R}: \psi, \partial_s \psi \in C^0(D_{a}) ~\mbox{with}~ \left\| \psi \right\|_a<\infty\} 
\text{,}
\notag
\end{align}
where 
\begin{align}
\left\| \psi \right\|_a:=\underset{(s,t)\in D_a}\sup\left\{|\psi(s,t)|+|\partial_s \psi(s,t)| \right\}
\text{.}
\notag
\end{align}
Furthermore, for the purpose of estimating the Lagrange multiplier, we define the space 
\begin{align}
\mathfrak{B}_{a}^{L}(d,M):=
\left\{ 
\psi \in\mathfrak{B}_{a} : 
\underset{t\in[0,a)}\inf \underset{[0,L]\times\{t\}}{\text{osc}} (\psi)\geq d, 
~\sup_{D_a} |\partial_s \psi|\leq M 
\right\}
\text{.} 
\notag
\end{align} 

For a given $\varphi_0\in C^1(\mathbb R)\cap L^{\infty}(\mathbb R)$ 
with $\underset{[0,L]}\sup\text{ }|\varphi_0^\prime|\le M$, 
we define the map $\mathcal{H}_{\varphi_0}$ from 
$\{\psi\in\mathfrak{B}_{a}^{L}(d,M): \psi(s,0)=\varphi_0\}\subset\mathfrak{B}_{a}$ 
to $\mathfrak{B}_{a}$ by 
\begin{align}
\mathcal{H}_{\varphi_0}(\psi)=U_{\varphi_0}+H^{\psi}_0 
=:H_{\varphi_0}^{\psi} 
\text{,}
\label{eq:H_phi^v}
\end{align} 
where $U_{\varphi_0}$ is the solution of 
\begin{align} 
\label{eq:U}
\begin{cases}
\partial_t U_{\varphi_0}=\partial_s^2 U_{\varphi_0} &\mbox{in}~ D_{a},
\\
U_{\varphi_0}(s,0)=\varphi_0(s), &\forall~s\in\mathbb{R}, 
\end{cases}
\end{align}
$H^{\psi}_0$ is the solution of
\begin{align} 
\label{eq:H}
\begin{cases}
\partial_t H^{\psi}_0=\partial_s^2 H^{\psi}_0+h(\psi) &\mbox{in}~D_{a},
\\
H^{\psi}_0 (s,0)=0, &\forall~s\in\mathbb{R}, 
\\
h(\psi):=\lambda_1(t)\sin \psi(s,t)-\lambda_2(t)\cos \psi(s,t), 
\end{cases}
\end{align}
and $\lambda_1(t)$, $\lambda_2(t)$ are determined by the formulae 
(\ref{eq:lambda-phi-Det}), (\ref{eq:lambda-phi-hingedBC}). 
Notice that   
$\{\psi\in\mathfrak{B}_{a}^{L}(d,M): \psi(s,0)=\varphi_0\}$ 
is a closed subset of the Banach space $(\mathfrak{B}_{a}, \left\|\cdot\right\|_a)$.

For a given $\psi\in\{\psi\in\mathfrak{B}_{a}^{L}(d,M): \psi(s,0)=\varphi_0\}$, 
the functions $\lambda_1(t)$ and $\lambda_2(t)$ in (\ref{eq:H}) are continuous in $t$. 
Thus, both (\ref{eq:U}) and (\ref{eq:H}) are linear parabolic equations with sufficiently regular coefficients 
such that we may apply Lemma \ref{lem:19.2.1_Cannon84} 
to obtain the existence of classical $C^2$-smooth solutions of 
(\ref{eq:U}) and (\ref{eq:H}) with the following expression 
\begin{align}
U_{\varphi_0}(s,t)&=\int_{-\infty}^{\infty}K(s-\xi,t)\varphi_0(\xi)\text{ }d\xi, 
~~~~~~~~~~~~~~~~~~~~ \text{ } (s,t)\in\mathbb R\times[0,\infty), 
\label{eq:U_varphi_0} 
\\
H^{\psi}_0 (s,t)
&=\int_0^t\int_{-\infty}^{\infty}K(s-\xi,t-\tau)h(\psi)(\xi,\tau) \text{ }d\xi d\tau, 
~ \text{ } (s,t)\in\mathbb R\times[0,\infty), 
\label{eq:H^varphi}
\end{align}
where $K(s,t):=\frac{1}{\sqrt{4\pi t}}e^{-\frac{s^2}{4t}}$. 
Therefore, the map 
$\mathcal{H}_{\varphi_0}:\{\psi\in\mathfrak{B}_{a}^{L}(d,M): 
\psi(s,0)=\varphi_0\}\rightarrow\mathfrak{B}_{a}$ 
is well-defined for any $a>0$.

\begin{lem} \label{lem:bd lambda}
Let $d\in(0,\pi/2)$ be a positive constant. 
Then for any $\psi\in\mathfrak{B}_{a}^{L}(d,M)$, 
there exist positive numbers 
$C_1$, $C_2$, $C_3$, depending only on $M$ and $d$, 
such that
\begin{align}
C_1\leq\int_0^L\cos^2\psi \text{ }ds~\mbox{, }~C_1
\leq\int_0^L\sin^2\psi \text{ }ds 
\text{,}
\label{ineq:c_and_s}
\end{align}
\begin{align}
C_2\leq\int_0^L\cos^2\psi \text{ }ds \cdot \int_0^L\sin^2\psi \text{ }ds
-\left(\int_0^L\cos\psi\sin\psi \text{ }ds\right)^2 
\text{,}
\label{ineq:c^s^2-(cs)^2}
\end{align}
\begin{align}
|\lambda_1|, ~ |\lambda_2|, ~ |h(\psi)|\leq C_3 
\text{.}
\label{ineq:bd h}
\end{align}
\end{lem}

\begin{proof}
From the assumption  
$\underset{t\in[0,a)}\inf \underset{[0,L]\times\{t\}}{\text{osc}} (\psi)\geq d>0$, for any fixed $t\in[0,a)$, there exists $s_0=s_0(t)\in[0,L]$ such that 
\[
|\cos \psi(s_0, t)| \ge\sin\frac{d}{2}
\text{.}
\]  
Since $\psi(\cdot,t)$ is $C^1$-smooth 
and fulfills $\underset{D_a}\sup\text{ } |\partial_s\psi| \leq M$, 
we may conclude that 
\[
\mathcal{L}\left(\left
\{\sigma\in[0,L]: |\psi(\sigma,t)-\psi(s_0,t)|\le \frac{d}{4} \right\}
\right)
\ge\frac{d}{4M} 
\text{,}
\] 
where $\mathcal{L}$ is the one-dimensional Lebesgue measure. 
Thus, 
\begin{align} 
\mathcal{L}\left(
\left\{\sigma\in[0,L]: |\cos\psi(\sigma,t)|\ge \sin\frac{d}{4}\right\}\right)
\ge\frac{d}{4M} 
\text{.}
\label{eq:lower_bdd_lambda-0} 
\end{align}
Therefore, 
\begin{align}
\int_0^L\cos^2\psi\text{ }ds
\geq 
\frac{d}{4M}\sin^2 \frac{d}{4}
=:C_1
\text{,}
\notag
\end{align}
which gives the first inequality in (\ref{ineq:c_and_s}). 

Similarly, the second inequality in (\ref{ineq:c_and_s}) also holds by applying the same argument. 

As 
$f=\frac{\cos\psi}{(\int^{L}_{0}\cos^2\psi \text{ } ds)^{1/2}}$,
$g=\frac{\sin\psi}{(\int^{L}_{0}\sin^2\psi \text{ } ds)^{1/2}}$, 
we may apply 
\[
\frac{1}{2}\int^{L}_{0}(f\pm g)^2\text{ }ds=1\pm\int^{L}_{0} f g\text{ }ds 
\] 
and let 
$\mu=\frac{(\int^{L}_{0}\cos^2\psi \text{ } ds)^{1/2}}
{(\int^{L}_{0}\sin^2\psi \text{ } ds)^{1/2}}$ 
to derive 
\begin{align}
&\int_{0}^{L}\cos^2\psi \text{ }ds\cdot\int_{0}^{L}\sin^2\psi \text{ }ds
-\left(\int_0^L\cos\psi\sin\psi \text{ }ds\right)^2 
\notag 
\\
=&
\left(\int_0^L\cos^2\psi \text{ }ds\cdot\int_0^L\sin^2\psi \text{ }ds \right) 
\left[1- \left(\int_0^L f g \text{ }ds\right)^2\right] 
\notag
\\
=&
\left(\int_0^L\cos^2\psi \text{ }ds\cdot\int_0^L\sin^2\psi \text{ }ds \right) 
\left(1-\int_0^L f g \text{ }ds\right)
\left(1+\int_0^L f g \text{ }ds\right)
\notag
\\
=&\frac{1}{4}\int_0^L\cos^2\psi \text{ }ds 
\cdot \int_0^L\sin^2\psi \text{ }ds 
\left(\int_0^L(f-g)^2 \text{ }ds\right)
\left(\int_0^L(f+g)^2 \text{ }ds\right) 
\notag 
\\
=&\frac{1}{4}\left(\int_0^L(\cos\psi-\mu\sin\psi)^2 \text{ }ds\right)
\left(\int_0^L\left(\frac{1}{\mu}\cos\psi+\sin\psi\right)^2\text{ }ds\right) 
\notag  
\\
=&\frac{1}{4}(1+\mu^2) (1+1/\mu^2)
\left(\int_0^L \sin^2(\psi+c_1(\mu)) \text{ }ds\right)
\left(\int_0^L \sin^2(\psi+c_2(\mu)) \text{ }ds\right)
\notag
\\
\geq&
\left(\int_0^L \sin^2(\psi+c_1(\mu)) \text{ }ds\right)
\left(\int_0^L \sin^2(\psi+c_2(\mu)) \text{ }ds\right) 
\text{.}
\label{eq:lower_bdd_lambda-1} 
\end{align}
By applying a similar argument as in deriving 
(\ref{eq:lower_bdd_lambda-0}), 
we obtain 
\begin{align} 
\mathcal{L}\left(
\left\{\sigma\in[0,L]: |\sin(\psi(\sigma,\cdot)+c_i(\mu))|
\ge \sin \frac{d}{4}\right\}
\right)
\ge\frac{d}{4M} 
\label{eq:lower_bdd_lambda-2} 
\end{align}
for $i\in\{1,2\}$. 
Thus, (\ref{ineq:c^s^2-(cs)^2}) is obtained by applying 
(\ref{eq:lower_bdd_lambda-1}) and (\ref{eq:lower_bdd_lambda-2}). 

The inequalities in (\ref{ineq:bd h}) are implied by applying 
(\ref{ineq:c^s^2-(cs)^2}), 
(\ref{eq:lambda-phi-Det}), (\ref{eq:lambda-phi-hingedBC}).
\end{proof}

\begin{lem} \label{lem:h}
Let $h:\mathbb R\rightarrow\mathbb R$ be the function given in (\ref{eq:H}). 
\begin{itemize}
\item [(i)]
For any $\psi_0, \psi_1\in \mathfrak{B}_{a}^{L}(d,M)$, one has 
\begin{align}
|h(\psi_1)-h(\psi_0)| 
\leq C_4 \cdot \left\| \psi_1-\psi_0 \right\|_a 
\label{esti:h}
\end{align}
for some constant $C_4=C_4(d, M)$. 
\item [(ii)] For a given $\varphi\in \mathfrak{B}_{a}^{L}(d,M)$ 
with the regular properties  
\[
\partial_s^j\varphi\in C^0(\mathbb R\times [0,a])\cap L^\infty(\mathbb R\times [0,a])
\text{, } \forall \text{ } j \in \{0,1,...,m\} \text{,}
\]
the composite function $h(\varphi)$ also fulfills 
\[
\partial_s^j h(\varphi)\in C^0(\mathbb R\times [0,a])\cap L^\infty(\mathbb R\times [0,a])
\text{, } \forall \text{ } j \in \{0,1,...,m\} \text{.}
\]
\end{itemize}
\end{lem}

\begin{proof}

$(i)$ 
From (\ref{eq:H}), we may write 
\[
h(\psi_\alpha)=
-\left<
\left(\int_I \text{ } (\partial_s\psi_\alpha)^2 T_{\psi_\alpha}\text{ }ds\right) 
A_{T_{\psi_\alpha}}^{-1}, 
T_{\psi_\alpha}^\perp
\right>
\text{.}
\] 
By using the formulae (\ref{eq:T=(cos,sin)}), (\ref{eq:matrix_A}), 
(\ref{eq:lambda-phi-Det}), (\ref{eq:lambda-phi-hingedBC}), 
and applying the assumption $\psi_0, \psi_1\in \mathfrak{B}_{a}^{L}(d,M)$, 
we obtain (\ref{esti:h}) from the following calculation 
\begin{align}
&\left|h(\psi_1)-h(\psi_0)\right| 
\notag
\\
&\le 
\left|
\left<
\left(\int_I \text{ } (\partial_s\psi_0)^2 T_{\psi_0}\text{ }ds\right) A_{T_{\psi_0}}^{-1}, 
T_{\psi_1}^\perp - T_{\psi_0}^\perp
\right>
\right|
\notag 
\\
& +
\left|
\left<
\left(\int_I \text{ } (\partial_s\psi_0)^2 T_{\psi_0}\text{ }ds\right) 
\left(A_{T_{\psi_1}}^{-1}-A_{T_{\psi_0}}^{-1}\right), 
T_{\psi_1}^\perp
\right>
\right|
\notag
\\ 
& +
\left|
\left<
\left(\int_I \text{ } (\partial_s\psi_0)^2 \cdot \left(T_{\psi_1}-T_{\psi_0}\right)\text{ }ds\right) 
A_{T_{\psi_1}}^{-1}, 
T_{\psi_1}^\perp
\right>
\right|
\notag
\\
& +
\left|
\left<
\left(\int_I \text{ } \left[(\partial_s\psi_1)^2 - (\partial_s\psi_0)^2\right]\cdot T_{\psi_1}\text{ }ds\right) 
A_{T_{\psi_1}}^{-1}, 
T_{\psi_1}^\perp
\right>
\right| 
\text{.}
\notag 
\end{align} 
We leave the details of the calculation to the reader for the sake of conciseness. 

$(ii)$ 
The conclusion $(ii)$ follows directly from the assumption and the formula 
\[
\partial_s^j h(\varphi)(s,t)=-\langle \vec\lambda(t), \partial_s^j T_{\varphi}^\perp (s,t)\rangle 
\text{,}
\] 
where $T_{\varphi}^\perp=(-\sin\varphi, \cos\varphi)$, $j\in\mathbb N\cup\{0\}$. 

\end{proof}

\begin{lem} \label{lem:contraction}
Assume 
$\varphi_0\in C^1(\mathbb R)\cap L^{\infty}(\mathbb R)$
with 
$\underset{\mathbb{R}}\sup \text{ } |\varphi^\prime_0|=M_0/2<\infty$ 
and 
$\underset{[0,L]}{\text{osc}} \text{ }\varphi_0= 2 d_0>0$. 

(i) There exists a positive number $b_0\in(0, a)$ 
so that the map given in (\ref{eq:H_phi^v}), 
\[
\mathcal{H}_{\varphi_0}: 
\{\psi\in\mathfrak{B}_{b}^{L}(d_0,M_0): 
\psi(s,0)=\varphi_0\}
\to
\{\psi\in\mathfrak{B}_{b}^{L}(d_0,M_0): \psi(s,0)=\varphi_0\}
\text{,}
\] 
is well-defined for all $0<b\le b_0$. 

(ii) 
There exists a positive time $t_0$ such that 
\[
\mathcal{H}_{\varphi_0}:
\{\psi\in\mathfrak{B}_{t_0}^{L}(d_0,M_0): \psi(s,0)=\varphi_0\}
\to
\{\psi\in\mathfrak{B}_{t_0}^{L}(d_0,M_0): \psi(s,0)=\varphi_0\}
\]
is a contraction map.
\end{lem}

\begin{proof}
To prove $(i)$, we need to verify that there exists a constant $b_0>0$ such that, 
for all $b\in(0,b_0)$ and 
$\psi\in\{\varphi\in\mathfrak{B}_{b}^{L}(d_0,M_0): \varphi(s,0)=\varphi_0\}$, 
\begin{align}
|\partial_s H_{\varphi_0}^\psi|\leq M_0~\mbox{in}~  D_b
~~\mbox{, }~~ 
\underset{t\in[0,b)}\inf \underset{[0,L]\times\{t\}}{\text{osc}} 
(H_{\varphi_0}^\psi)\geq d_0  
\text{.}
\notag 
\end{align}
Since $\varphi_0$ is bounded and continuous on $\mathbb R$, 
from the standard argument on the initial value problem for the heat equation on 
$\mathbb R$ 
(e.g., see Chapter $3$ of \cite{Cannon84}), the function 
$U_{\varphi_0}$, given in (\ref{eq:U_varphi_0}), belongs to the class of 
$C^{\infty}(\mathbb R\times(0,\infty))\cap C^0(\mathbb R\times[0,\infty))$. 
Observe that 
\begin{align}
\partial_s U_{\varphi_0}(s,t)
&=
\int_{-\infty}^{\infty}\partial_s K(s-\xi,t)\varphi_0 (\xi)\text{ }d\xi
=\int_{-\infty}^{\infty} -\partial_\xi K(s-\xi,t)\varphi_0 (\xi)\text{ }d\xi 
\notag
\\
&=\int_{-\infty}^{\infty} K(s-\xi,t)\varphi_0^\prime (\xi)\text{ }d\xi 
\text{,}
\label{eq:d_s U}
\end{align}
where the last equality comes from applying the assumption of $\varphi_0$ in (\ref{eq:varphi_0_C^infty}). 
Thus, using the same argument for $\partial_s U_{\varphi_0}$, 
$\partial_s U_{\varphi_0}$ also belongs to the class 
$C^{\infty}(\mathbb R\times(0,\infty))\cap C^0(\mathbb R\times[0,\infty))$. 
From (\ref{eq:d_s U}) and the assumption $\underset{\mathbb R}\sup\|\varphi_0^\prime\|\le M_0/2$ 
and the identity 
\begin{align}
\int_{-\infty}^{\infty} K(x,t)\text{ }dx=1 
\text{,}
\label{int:K}
\end{align} 
we conclude 
\begin{align}
|\partial_s U_{\varphi_0} (s,t)|\leq \frac{M_0}{2} 
\label{ineq:U-1}
\end{align}
for all $s\in\mathbb R$, $t\ge 0$. 
In fact, by applying this argument step by step, we obtain 
\begin{align}
\underset{t\rightarrow +0}\lim \partial_s^j U_{\varphi_0}(s,t)=\varphi_0^{(j)} (s) 
\label{eq:U_varphi=smooth}
\end{align} 
for all $s\in\mathbb R$, and any $j\in\mathbb N\cup\{0\}$.

To show the continuity of $H^\psi_0$ and $\partial_s H^\psi_0$ up to the parabolic boundary 
$\mathbb R\times\{0\}$, 
observe that, by applying Lemma \ref{lem:bd lambda} and (\ref{int:K}), 
these hold for all $s\in\mathbb R$ 
\begin{align}
|H^\psi_0 (s,t)|
\leq 
\int_0^t\int_{-\infty}^{\infty}|K(s-\xi,t-\tau)||h(\psi)(\xi,\tau)| \text{ }d\xi d\tau
\leq t \cdot C_3 
\text{.}
\label{eq:d_s^0(H)}
\end{align}
$\forall \text{ } s \in \mathbb R$. 
By applying the identity  
\begin{align}
\int_{-\infty}^{\infty}\partial_x K(x,t)\text{ }dx=\frac{-1}{\sqrt{\pi t}} 
\text{,}
\label{int:d_s(K)}
\end{align}
we obtain   
\begin{align}
|\partial_s H^\psi_0 (s,t)|
\leq\int_0^t\int_{-\infty}^{\infty}
|\partial_s K(s-\xi,t-\tau)||h(\psi)(\xi,\tau)| \text{ }d\xi d\tau
\leq 
2\frac{\sqrt{t}}{\sqrt{\pi}}\cdot C_3 
\text{.}
\label{eq:d_s(H)}
\end{align}
Thus, we conclude that for all $s \in\mathbb R $ 
\begin{align}
\underset{t\rightarrow +0}\lim H_{0}^{\psi}(s,t)=0  
~~\mbox{ and }~~
\underset{t\rightarrow +0}\lim \partial_s H_{0}^{\psi}(s,t)=0. 
\label{eq:H^psi--->0} 
\end{align}
Hence by  
(\ref{ineq:U-1}), (\ref{eq:d_s^0(H)}), (\ref{eq:d_s(H)}), (\ref{eq:H^psi--->0}), 
$H_{\varphi_0}^{\psi}$ and $\partial_s H_{\varphi_0}^{\psi}$ are continuous up to $\mathbb R\times\{0\}$. 
Therefore, we may choose a sufficiently small $\alpha>0$ 
such that 
\begin{align}
|\partial_s H_{\varphi_0}^{\psi}(s,t)| 
\leq M_0, ~~\forall\text{ } (s,t)\in D_{\alpha} 
\text{.}
\notag
\end{align}
Moreover, 
since $H_{\varphi_0}^{\psi}$ is continuous on the closure of $D_t^L$ 
(which is compact for any $t>0$), 
$\underset{t\rightarrow +0}\lim \text{ }H_{\varphi_0}^{\psi}(s,t)=\varphi_0(s)$, 
and $\underset{[0,L]}{\text{osc}}\text{ } \varphi_0 \ge 2 d_0$, 
we may choose a sufficiently small $\beta>0$ such that 
\begin{align}
\underset{t\in[0,\beta)}\inf \underset{[0,L]\times\{t\}}{\text{osc}} 
(H_{\varphi_0}^{\psi}) 
\geq d_0 
\text{.}
\notag
\end{align}
The conclusion $(i)$ is now obtained by letting $b_0:=\min\{\alpha, \beta\}$.

$(ii)$ 
From $(i)$, we let $t_0\in(0,b_0]$ below. 
By applying Lemma \ref{lem:h}, (\ref{int:K}), and (\ref{int:d_s(K)}), 
we obtain that for all $(s,t)\in\mathbb R\times(0,t_0)$ 
\begin{align}
&|H_{\varphi_0}^{\psi_1}(s,t)-H_{\varphi_0}^{\psi_0}(s,t)| 
\notag 
\\
\leq&\int_0^{t_0}\int_{-\infty}^{\infty}|K(s-\xi,t-\tau)| 
\cdot |h(\psi_1)(\xi,\tau)-h(\psi_0)(\xi,\tau)| \text{ }d\xi d\tau 
\notag 
\\ 
\leq& C_4 t_0 \left\|\psi_1-\psi_0\right\|_{t_0} 
\label{ineq:H1}
\end{align}
and
\begin{align}
&|\partial_s H_{\varphi_0}^{\psi_1}(s,t)
-\partial_s H_{\varphi_0}^{\psi_0}(s,t)| 
\notag 
\\
\leq&\int_0^{t_0}\int_{-\infty}^{\infty}|\partial_s K(s-\xi,t-\tau)| 
\cdot |h(\psi_1)(\xi,\tau)-h(\psi_0)(\xi,\tau)|
\text{ } d\xi d\tau 
\notag 
\\
\leq&2\frac{\sqrt{t_0}}{\sqrt{\pi}} C_4 \left\|\psi_1-\psi_0\right\|_{t_0}. 
\label{ineq:H2}
\end{align}
From (\ref{ineq:H1}) and (\ref{ineq:H2}), we deduce 
\begin{align}
\left\| H_{\varphi_0}^{\psi_1}- H_{\varphi_0}^{\psi_0} \right\|_{t_0}
\leq C_4\left(t_0+2\frac{\sqrt{t_0}}{\sqrt{\pi}}\right) 
\left\|\psi_1-\psi_0\right\|_{t_0}. 
\notag
\end{align}
Thus, conclusion $(ii)$ is obtained by choosing a small $t_0\le b_0$ such that 
\begin{align}
C_4\left(t_0+2\frac{\sqrt{t_0}}{\sqrt{\pi}}\right)<1 
\text{.}
\notag
\end{align}
\end{proof}

\begin{proof}[Proof of Proposition \ref{prop:Short-T.E.}] 

There are two steps in the proof. 

Step $1$: existence of classical solutions and the integral representation. 

Define the sequence of functions
\begin{align}
\psi_n=
\begin{cases}
\varphi_0    &n=0, 
\\
\mathcal{H}_{\varphi_0}(\psi_{n-1})  &n\ge 1, 
\end{cases}
\label{eq:iteration_seq}
\end{align}
and $M_0$, $d_0$ are given as before. 
From Lemma \ref{lem:contraction}, 
the sequence \{$\psi_n$\} is a Cauchy sequence in 
$\mathfrak{B}_{t_0}^{L}(d_0,M_0)$, 
which is a closed subset in the Banach space 
$\mathfrak{B}_{t_0}$. 
Thus, there exists a unique function,
$\varphi\in\mathfrak{B}_{t_0}^{L}(d_0,M_0)$, such that
$\psi_n\to\varphi$ in the topology of 
$(\mathfrak{B}_{t_0},\left\|\cdot\right\|_{t_0})$. 
On the other hand, for each $n\ge 0$, the function 
$\psi_{n+1}$ satisfies the following equation
\[
\begin{cases}
\partial_t \psi_{n+1}=\partial^2_s \psi_{n+1}+h(\psi_n)
~&\mbox{for}~(s,t)\in D_{t_0},
\\
\psi_{n+1}(s,0)=\varphi_0(s) &\mbox{for}~  s\in\mathbb R.
\end{cases}
\]
Therefore, as $n\rightarrow\infty$, the function $\varphi$ satisfies
equation (\ref{eq:grad.flow.G}) and can be written as
\begin{align}
\varphi(s,t)&=U_{\varphi_0}(s,t)+H_0^\varphi (s,t) 
\notag 
\\
&=U_{\varphi_0}(s,t)+
\int^t_0\int_{-\infty}^{\infty}K(s-\xi,t-\tau)h(\varphi)(\xi,\tau)\text{ }d\xi d\tau 
\text{.}
\label{eq:integral_formula_soln}
\end{align}

Step $2$: smoothness of the right hand side of (\ref{eq:integral_formula_soln}). 



From the regularity assumption of $\varphi_0$ on (\ref{eq:varphi_0_C^infty}) 
and applying integration by parts inductively, 
one obtains 
\[
\partial_s^\ell U_{\varphi_0}(s,t)=\int_{-\infty}^{\infty}K(s-\xi,t)\cdot \varphi_0^{(\ell)}(\xi)
\text{ }d\xi 
\]
for all $\ell\in\mathbb N$. 
Thus, we conclude 
\begin{align}
\partial_s^\ell U_{\varphi_0}\in C^0 (\mathbb R\times [0,+\infty)), 
~~\forall \text{ }\ell\in\mathbb N\cup\{0\}
\text{,}
\label{eq:U_varphi=smooth-1} 
\end{align}
by the same argument originally applied to $U_{\varphi_0}$ before. 

To finish the proof, we only need to work with the last term of 
(\ref{eq:integral_formula_soln}). 
First note that from Lemma \ref{lem:h} $(ii)$ 
the regularity of $h(\varphi)$ is the same as that of $\varphi$. 
Since $\varphi\in \mathfrak{B}_{t_0}$, 
$h(\varphi)$ also belongs to $\mathfrak{B}_{t_0}$. 
Thus, by applying Lemma \ref{lem:19.2.1_Cannon84}, 
$\partial_s^\ell H_0^\varphi\in C^0 (\mathbb R\times[0,t_0])$, 
for each $\ell=\{0,1,2\}$. 
Together with (\ref{eq:U_varphi=smooth-1}), we have that 
$\partial_s^\ell H_{\varphi_0}^\varphi$ 
(i.e., $\partial_s^\ell \varphi$) belongs to $C^0 (\mathbb R\times[0,t_0])$, 
$\forall\text{ }\ell=\{0,1,2\}$. 
By applying Lemma \ref{lem:h} $(ii)$ again, $h(\varphi)$ now also fulfills 
$\partial_s^\ell h(\varphi)\in C^0 (\mathbb R\times[0,t_0])$, 
for each $\ell=\{0,1,2\}$.  
Now,  
by applying the formula, 
\begin{align}
\partial_s^\ell H_0^\varphi (s,t)
=&\int_0^t\int_{-\infty}^{\infty}\partial_s K(s-\xi,t-\tau)\cdot \partial_\xi^{\ell-1} h(\varphi)(\xi,\tau)\text{ }d\xi d\tau 
\notag
\\
=&\int_0^t\int_{-\infty}^{\infty} K(s-\xi,t-\tau)\cdot \partial_\xi^\ell h(\varphi)(\xi,\tau)
\text{ }d\xi d\tau 
\text{,}
\notag 
\end{align}
Lemma \ref{lem:19.2.1_Cannon84}, and Lemma \ref{lem:h}, 
we prove inductively that 
\begin{align}
\partial_s^\ell H_0^\varphi\in C^0 (\mathbb R\times[0,t_0]), 
~~ \forall \text{ } \ell\in\mathbb N\cup\{0\}
\label{eq:H_0^varphi=smooth-1}
\end{align} 
by a boot-strapping argument. 

The proof is finished by combining (\ref{eq:integral_formula_soln}), 
(\ref{eq:U_varphi=smooth-1}), and (\ref{eq:H_0^varphi=smooth-1}). 
\end{proof}

\begin{proof}[Proof of Theorem \ref{thm:STE_for_N}] 

We may extend $\varphi_0$ as an even function from $[0,L]$ to $[-L,L]$. 
Then we extend $\varphi_0$ as a periodic function with period $2L$ from $[-L,L]$ 
to the whole real line $\mathbb R$.
Denote this extended function by $\widetilde{\varphi}_0$. 
From the assumption of vanishing of derivatives of $\varphi_0$ at the end points, 
we deduce that there exists a bounded sequence in $\ell$ of numbers 
$\{M_\ell\}$, $\ell\in\{0,1,2,...\}$, such that 
$\widetilde{\varphi}_0\in C^\infty(\mathbb R)$ and it derivatives fulfill 
$\underset{\mathbb R}\sup\left\{|\widetilde{\varphi}_0^{(\ell)}|\right\}\le M_\ell$ 
$\forall\text{ }\ell\in\mathbb N\cup\{0\}$. 
Therefore, from Proposition \ref{prop:Short-T.E.}, 
there is a positive number $t_0$ and a $C^\infty$-smooth function
$\varphi$ satisfying
\begin{align}
\begin{cases}
\frac{\partial}{\partial t}\varphi(s,t)
=\frac{\partial^2}{\partial s^2}\varphi(s,t)
+\lambda_1(t) \sin\varphi(s,t)-\lambda_2(t)\cos\varphi(s,t)
~~&\mbox{in}~~D_{t_0}, 
\\
\varphi(s,0)=\widetilde{\varphi}_0(s) &\mbox{on}~~\mathbb R\times\{0\}.
\end{cases}\notag
\end{align}
Since $\widetilde{\varphi}_0(s)$ is the extension of $\varphi_0(s)$, 
this $\varphi$ fulfills 
the first two equations in (\ref{eq:grad.flow.H}).

To finish the proof, we only need to show that this $\varphi$ also 
fulfills the boundary conditions in 
(\ref{eq:grad.flow.H}), 
i.e., 
\begin{align}
\partial_s\varphi(b,t)=k(b,t)=0, \forall\text{ }b\in\partial I, t\in (0,t_0) 
\text{.}
\label{eq:k(b,t)=0}
\end{align} 
We prove (\ref{eq:k(b,t)=0}) by showing that each term in 
$\{h(\psi_n)\}_{n\in\mathbb N\cup\{0\}}$ 
is an even function of $s$ with period $2L$, 
where 
$\{\psi_n\}_{n\in\mathbb N\cup\{0\}}$ 
the iteration sequence in (\ref{eq:iteration_seq}). 
Note that $\{\psi_n\}_{n\in\mathbb N\cup\{0\}}$ 
is a Cauchy sequence, which converges to $\varphi$ 
in the Banach space $(\mathfrak{B}_{t_0},||\cdot||_{t_0})$. 
Since for a fixed $t$, 
$\psi_0(s,t)=\widetilde{\varphi}_0(s)$ is an even function of $s$ with period $2L$, 
from the formulae (\ref{eq:U_varphi_0}) and (\ref{eq:H^varphi}), 
it is easy to verify that $\psi_1$ is also an even function of $s$ with period $2L$. 
Thus, by an induction argument, we conclude that every function in the sequence 
$\{\psi_n\}$ is an even function of $s$ with period $2L$. 
Since $\{\psi_n\}$ converges to $\varphi$ in the 
Banach space $(\mathfrak{B}_{t_0},||\cdot||_{t_0})$, 
we conclude that 
$\varphi\in\mathfrak{B}_{t_0}$ is also an even function of $s$ with period $2L$. 
Now, the vanishing condition in (\ref{eq:k(b,t)=0}) is a result of 
this property of $\varphi$.
In other words, for fixed $t\in(0,t_0)$, 
as $\varphi(\cdot,t)\in C^1(\mathbb R)$ is an even function with period $2L$, 
then $\partial_s \varphi(0,t)=0=\partial_s \varphi(L,t)$. 


\end{proof}

\subsection{Long time existence and asymptotics}

\begin{lem}
\label{lem:nabla_s^2m(kappa)=0_at_bdry}
Assume the curvature vector $\kappa$ remains smooth during the $L^2$-flow 
(\ref{eq:flow-T}) up to $t=t_0>0$. 
Then, for all $t\in(0,t_0)$ and $\ell\in\mathbb N\cup\{0\}$, 
\begin{equation*}
\partial_s^{2\ell}\kappa (b, t)=0=\nabla_s^{2\ell}\kappa (b, t), 
\text{ } \forall \text{ }b\in\partial I, \text{ } 
\forall 
\text{ }
t\in[0,t_0) 
\text{.}
\end{equation*} 
\end{lem}

\begin{proof}

$(i) \text{ }\text{ to show }\partial_s^{2\ell}\kappa (b, \cdot)=0$:  

From a direct computation, we have 
\[
\partial_s^{2}\kappa
=\nabla_s^2\kappa -3\langle \nabla_s\kappa, \kappa \rangle T 
-|\kappa|^2\kappa 
\text{.}
\] 
Together with the hinged boundary condition (\ref{eq:bc-hinged}), 
we derive $\partial_s^{2}\kappa(b, t)=0$, $\forall \text{ }t\in[0,t_0)$. 
To show $\partial_s^{2\ell}\kappa (b, t)=0$, 
$\forall \text{ }t\in[0,t_0), \forall\text{ }\ell\ge 2$, 
we argue inductively by the following calculation:  
\begin{equation*}
\begin{array}{l}
\partial_t \partial_s^{2\ell-2}\kappa 
= \partial_s^{2\ell-2}\partial_t \kappa 
= \partial_s^{2\ell-2}(\partial_s^2 \kappa 
+ \langle \vec\lambda, T \rangle \kappa
+ |\kappa|^2\kappa
+2\langle \partial_s\kappa, \kappa \rangle T
+\langle \vec\lambda, \kappa \rangle T) 
\\ \\ 
= \partial_s^{2\ell} \kappa 
+\partial_s^{2\ell-2}
\left(
\langle \vec\lambda, T \rangle \kappa+|\kappa|^2\kappa
+2\langle \partial_s\kappa, \kappa \rangle\cdot T
+\langle \vec\lambda, \kappa \rangle\cdot T
\right) 
\\ \\ 
= \partial_s^{2\ell} \kappa 
+\underset{i+j=2\ell-2}\sum C_{1}(i,j) 
\langle \vec\lambda, \partial_s^{i}T \rangle \cdot \partial_s^{j} \kappa 
+\underset{i+j+k=2\ell-2}\sum C_{2}(i,j,k) 
\langle\partial_s^{i}\kappa, \partial_s^{j}\kappa\rangle\cdot\partial_s^{k}\kappa 
\\ \\
+\underset{i+j+k=2\ell-2}\sum C_{3}(i,j,k) 
\langle \partial_s^{i+1}\kappa, \partial_s^{j}\kappa \rangle\cdot \partial_s^{k}T
+\underset{i+j=2\ell-2}\sum C_{4}(i,j) 
\langle \vec\lambda, \partial_s^{i}\kappa \rangle\cdot \partial_s^{j}T 
\text{,}
\end{array}
\end{equation*}
where (\ref{eq:flow-T}) is applied to derive the second equality with constants 
$C_{1}(i,j)$, $C_{2}(i,j,k)$, $C_{3}(i,j,k)$, $C_{4}(i,j)$.  
By applying the assumption 
\[
\partial_s^{2i}\kappa(b, t)=0, 
\text{ } 
\forall 
\text{ }
i\in\{0, 1,..., \ell-1\}, 
\text{ } 
\forall 
\text{ }
b\in\partial I, 
\text{ } 
\forall 
\text{ }
t\in[0,t_0) 
\text{,}
\]  
the term $\partial_t \partial_s^{2\ell-2}\kappa$ 
and the last four terms on the R.H.S. all vanish. 
This gives 
$\partial_s^{2\ell}\kappa(b, \cdot)=0, \text{ } \forall \text{ }b\in\partial I$, 
and thus completes the proof of the first part.

$(ii) \text{ }\text{ to show }\nabla_s^{2\ell}\kappa (b, \cdot)=0$:  

Applying the hinged boundary condition (\ref{eq:bc-hinged}) to 
(\ref{eq:flow-kappa}) gives 
\[
\nabla_s^2\kappa(b, \cdot)=0, \text{ } \forall \text{ }b\in\partial I 
\text{.}
\] 
To show $\nabla_s^{2\ell}\kappa (b, \cdot)=0, \forall\text{ }\ell\ge 2$, 
we argue inductively from applying the result, 
$\partial_s^{2\ell}\kappa (b, \cdot)=\partial_s^{2\ell+1}T (b, \cdot)=0$, 
in part $1^\circ$ 
and the following calculation  
\begin{equation*}
\begin{array}{l}
\nabla_t \nabla_s^{2\ell-2}\kappa = \nabla_s^{2\ell-2}\nabla_t \kappa 
= \nabla_s^{2\ell-2}(\nabla_s^2 \kappa +\langle \vec\lambda, T \rangle \kappa)
= \nabla_s^{2\ell} \kappa 
+\nabla_s^{2\ell-2}(\langle \vec\lambda, T \rangle \kappa) 
\\ \\ 
= \nabla_s^{2\ell} \kappa 
+\underset{i+j=2\ell-2}{\sum}
C_{5}(i, j)\langle \vec\lambda, \partial_s^{i}T \rangle\cdot\nabla_s^{j}\kappa 
\text{,}
\end{array}
\end{equation*}
where $C_{5}(i, j)$ is a constant.

\end{proof}

\begin{lem}
\label{lem:induction-2}
Assume the curvature vector $\kappa$ remains smooth during the $L^2$-flow 
(\ref{eq:flow-T}) up to $t=t_0>0$. 
Then, $\forall \text{ } t\in(0,t_0)$ and $\forall\text{ }m\in\mathbb N$, 
\begin{equation*}
\frac{d}{dt}\frac{1}{2}
\underset{I}{\int}\text{ }|\nabla_s^m\kappa|^2\text{ }ds
+ \underset{I}{\int}\text{ }|\nabla_s^{m+1}\kappa|^2\text{ }ds 
= 
-\underset{I}{\int}\text{ }
\langle
\nabla_s^{m+1}\kappa, 
\nabla_s^{m-1}(\langle \vec\lambda, T\rangle\cdot\kappa) 
\rangle
\text{ }ds 
\text{.}
\end{equation*}
\end{lem}

\begin{proof}
In the argument below, 
we apply Lemma \ref{lem:nabla_s^2m(kappa)=0_at_bdry} and (\ref{eq:Lin_Lemma8-like-1}). 

Case $1 ~(\text{as } m=2\ell)$: 

\begin{equation*}
\begin{array}{l}
\frac{d}{dt}\frac{1}{2}
\underset{I}{\int}\text{ }|\nabla_s^{2\ell}\kappa|^2\text{ }ds 
= \underset{I}{\int}\text{ }
\langle
\nabla_s^{2\ell}\kappa, \nabla_t\nabla_s^{2\ell}\kappa 
\rangle
\text{ }ds 
= \underset{I}{\int}\text{ }
\langle
\nabla_s^{2\ell}\kappa, \nabla_s\nabla_t\nabla_s^{2\ell-1}\kappa 
\rangle
\text{ }ds 
\\ \\ 
= -\underset{I}{\int}\text{ }
\langle
\nabla_s^{2\ell+1}\kappa, \nabla_t\nabla_s^{2\ell-1}\kappa 
\rangle
\text{ }ds 
= -\underset{I}{\int}\text{ }
\langle
\nabla_s^{2\ell+1}\kappa, \nabla_s^{2\ell-1}\nabla_t\kappa 
\rangle
\text{ }ds
\\ \\
=
- \underset{I}{\int}\text{ }
\langle
\nabla_s^{2\ell+1}\kappa, 
\nabla_s^{2\ell-1}
(\nabla^2_s\kappa +\langle \vec\lambda, T\rangle\cdot\kappa) 
\rangle
\text{ }ds
\\ \\
=
-\underset{I}{\int}\text{ }
\langle
|\nabla_s^{2\ell+1}\kappa|^2 \text{ }ds
-\underset{I}{\int}\text{ }
\langle
\nabla_s^{2\ell+1}\kappa, 
\nabla_s^{2\ell-1}
(\langle \vec\lambda, T\rangle\cdot\kappa) 
\rangle
\text{ }ds 
\text{.}
\end{array}
\end{equation*}

Case $2 ~(\text{as } m=2\ell+1)$: 

Notice that, from Lemma \ref{lem:induction-2}, 
we have $\nabla_t\nabla^{2\ell}_s\kappa=0$ at the boundary. 
We obtain 
\begin{equation*}
\begin{array}{l}
\frac{d}{dt}\frac{1}{2}
\underset{I}{\int}\text{ }|\nabla_s^{2\ell+1}\kappa|^2\text{ }ds 
= \underset{I}{\int}\text{ }
\langle
\nabla_s^{2\ell+1}\kappa, \nabla_t\nabla_s^{2\ell+1}\kappa 
\rangle
\text{ }ds 
= \underset{I}{\int}\text{ }
\langle
\nabla_s^{2\ell+1}\kappa, \nabla_s\nabla_t\nabla_s^{2\ell}\kappa 
\rangle
\text{ }ds 
\\ \\ 
= -\underset{I}{\int}\text{ }
\langle
\nabla_s^{2\ell+2}\kappa, \nabla_t\nabla_s^{2\ell}\kappa 
\rangle
\text{ }ds 
= -\underset{I}{\int}\text{ }
\langle
\nabla_s^{2\ell+2}\kappa, \nabla_s^{2\ell}\nabla_t\kappa 
\rangle
\text{ }ds
\\ \\
=
- \underset{I}{\int}\text{ }
\langle
\nabla_s^{2\ell+2}\kappa, 
\nabla_s^{2\ell}
(\nabla^2_s\kappa +\langle \vec\lambda, T\rangle\cdot\kappa) 
\rangle
\text{ }ds
\\ \\
=
-\underset{I}{\int}\text{ }
\langle
|\nabla_s^{2\ell+2}\kappa|^2 \text{ }ds
-\underset{I}{\int}\text{ }
\langle
\nabla_s^{2\ell+2}\kappa, 
\nabla_s^{2\ell}
(\langle \vec\lambda, T\rangle\cdot\kappa) 
\rangle
\text{ }ds 
\text{.}
\end{array}
\end{equation*}

\end{proof}

\begin{proof}[Proof of Theorem \ref{thm:main}]

The long time existence is derived by a contradiction argument. 
Namely, assume the solution fails to be $C^\infty$-smooth 
at $t=t_{0}>0$. 
Then, we show that the 
$L^{2} $-norm of the derivatives of the curvature of any order remain uniformly
bounded for any $t<t_{0}$. 
This gives the contradiction. 
In the end of proof, 
we show the asymptotic behaviour for a convergent 
subsequence of the solutions up to the translation. 

For any $m\in \mathbb N$, 
we will obtain the long time existence by deriving the following estimates for  curvature integrals, 
\begin{equation}
\frac{d}{dt}\underset{I}{\int }|\nabla_s^m\kappa|^2\text{ }ds
+\varepsilon \cdot \underset{I}{\int }|\nabla_s^{m+1}\kappa|^2\text{ }ds
\leq C(\left\| \kappa_{0}\right\| _{L^{2}}, L, |\triangle p|, m)
\text{,}
\label{eq:mainesti}
\end{equation}
where $\varepsilon >0$ is a sufficiently small constant. 
By Gronwall inequality, (\ref{eq:mainesti}) implies 
\begin{align}
&\int_I |\nabla_{s}^{m}\kappa|^2\text{ }ds 
\le C(\Vert \kappa_{0} \Vert_{L^2}, L, |\triangle p|, m) 
\text{ , } \forall \text{ }  t\in(0,t_0) \text{,}
\label{eq:uniform_bdd_nabla^m}
\end{align}
where $\kappa_0=\kappa(s,0)$ is the curvature vector of the initial curve $f_0$.

Note that (\ref{eq:energy_ID-1}) gives the non-increasing property of the energy 
$\mathcal{F}_{L,\triangle p}\left[ T \right]$. 
Thus, as long as the smooth solutions of the $L^2$-flow exist 
$\forall\text{ } t\in (0,t_0)$, 
one has 
$\mathcal{F}_{L,\triangle p}(T)\leq \mathcal{F}_{L,\triangle p}(T_0)$, 
where $T_0$ is the tangent vector of initial curves. 
Therefore, $\forall\text{ }t\in(0,t_0)$, 
$\Vert \kappa (\cdot,t)\Vert_{L^2}^2 
\leq 2 \mathcal{F}_{L,\triangle p}(T_0)$.

From Lemma \ref{lem:induction-2}, 
we have 
\begin{align}
&
\frac{d}{dt}\frac{1}{2}
\underset{I}{\int}\text{ }|\nabla_s^m\kappa|^2\text{ }ds
+ \underset{I}{\int}\text{ }|\nabla_s^{m+1}\kappa|^2\text{ }ds 
= 
-\underset{I}{\int}\text{ }
\langle
\nabla_s^{m+1}\kappa, 
\nabla_s^{m-1}(\langle \vec\lambda, T\rangle\cdot\kappa) 
\rangle
\text{ }ds 
\notag
\\
\le&
\left(\underset{I}{\int}\text{ }|\nabla_s^{m+1}\kappa|^2 \text{ }ds
\right)^{1/2}
\cdot
\left(\underset{I}{\int}\text{ }
|\nabla_s^{m-1}(\langle \vec\lambda, T\rangle\cdot\kappa)|^2
\text{ }ds
\right)^{1/2} 
\text{.}
\notag
\end{align}
From applying (\ref{eq:Lin_Lemma8-like-4}), 
the uniform bound of $|\vec\lambda|$ in Lemma \ref{lem:lambda}, 
and Gagliardo-Nirenberg inequalities to the equation 
\begin{align}
&
\nabla_s^{m-1}(\langle \vec\lambda, T\rangle\cdot\kappa)
=
\sum_{i+j=m-1}
C_{6}(i, j)\cdot\langle \vec\lambda, \partial_s^{i}T \rangle\cdot\nabla_s^{j}\kappa 
\text{,}
\notag
\end{align}
we derive 
\begin{align}
&
\left(\underset{I}{\int}\text{ }
|\nabla_s^{m-1}(\langle \vec\lambda, T\rangle\cdot\kappa)|^2
\text{ }ds 
\right)^{1/2}
\notag
\le&
C(\left\|\kappa_0 \right\| _{L^2}, L,|\triangle p|, m)\cdot 
\left\|\kappa \right\|_{m+1, 2}^{(m-1)/(m+1)} 
\text{.}
\notag
\end{align} 
For any positive integer $\ell\in\{1, 2,..., m+1\}$, 
we apply (\ref{eq:Lin_Lemma8-like-4}) and Gagliardo-Nirenberg inequalities 
again to derive 
\begin{equation}
\left|
\underset{I}{\int}\text{ }|\nabla_s^{\ell}\kappa|^2\text{ }ds
-
\underset{I}{\int}\text{ }|\partial_s^{\ell}\kappa|^2\text{ }ds
\right|
\le 
C(\left\|\kappa_0 \right\| _{L^2}, L,|\triangle p|, \ell)\cdot 
\left\|\kappa \right\|_{m+1, 2}^{\gamma} 
\label{eq:difference-(partial^m-nabla^m)^2}
\end{equation}
for some $\gamma\in(0,2)$. 
Thus, by applying Gagliardo-Nirenberg inequalities, we have 
\begin{align}
&
\frac{d}{dt}
\underset{I}{\int}\text{ }|\nabla_s^m\kappa|^2\text{ }ds
+ \underset{I}{\int}\text{ }|\partial_s^{m+1}\kappa|^2\text{ }ds 
+ \underset{I}{\int}\text{ }|\nabla_s^{m+1}\kappa|^2\text{ }ds 
\notag
\\
\le & 
\epsilon\cdot\left\|\kappa \right\|_{m+1, 2}^{2}
+
C(\epsilon, \left\|\kappa_0 \right\| _{L^2}, L,|\triangle p|, m) 
\notag
\end{align}
for some sufficiently small $\epsilon>0$. 
Note that the vanishing conditions of $\partial_s^{2\ell}\kappa$ 
in Lemma \ref{lem:nabla_s^2m(kappa)=0_at_bdry} 
allow us to apply the Poincare inequality  
\[
\underset{I}{\int}\text{ }|\partial_s^{m+1}\kappa|^2\text{ }ds 
\ge 
C(L) \cdot 
\underset{I}{\int}\text{ }|\partial_s^{m}\kappa|^2\text{ }ds
\] 
for some $C(L)>0$, $m\in\mathbb N\cup\{0\}$. 
Therefore, by applying (\ref{eq:difference-(partial^m-nabla^m)^2}) again to 
the case $\ell=m$ and choosing a sufficiently small $\epsilon>0$, 
we obtain 
\begin{align}
\frac{d}{dt}
\underset{I}{\int}\text{ }|\nabla_s^m\kappa|^2\text{ }ds
+\frac{C(L)}{2}\cdot\underset{I}{\int}\text{ }|\nabla_s^{m}\kappa|^2\text{ }ds 
\le 
C(\left\|\kappa_0 \right\|_{L^2}, L,|\triangle p|, m) 
\text{.}
\notag
\end{align}  
Finally we apply Gronwall inequality to derive from this differential inequality 
a uniform bound of $\left\|\nabla_s^{m}\kappa \right\|_{L^2}$ as presented in 
(\ref{eq:uniform_bdd_nabla^m}).

Notice that a planar curve can be uniquely determined by its signed curvature $k$, 
up to translation and rotation. 
Moreover, 
the formulae to connect the curvature vector $\kappa$ 
and signed curvature $k$ hold, i.e., 
\[
\nabla _{s}^{m}\kappa = \partial _{s}^{m}k \cdot T^\perp, 
~~ \forall \text{ } m\in\mathbb N\cup\{0\}
\text{.}
\] 
Besides, we have 
\begin{equation*}
\left\| \partial _{s}^{m-1}k\right\| _{L^{\infty }}
\leq C\cdot \left\|
\partial _{s}^{m}k\right\| _{L^{1}},
\forall \text{ } m\in \mathbb N
\text{.}
\end{equation*}
Therefore, 
$\left\|\partial_s^{m}k \right\|_{L^2}$ is also uniformly bounded for all 
$m\in \mathbb N$. 
Now the proof of the long time existence is completed 
by a contradiction argument.

To analyze the asymptotic behaviour of the flow, we choose a subsequence of
$T \left( t,\cdot \right)$, which converges smoothly to $T_{\infty }(\cdot)$ 
as $t\rightarrow \infty$. 
In addition, we rewrite the energy identity in (\ref{eq:energy_ID-1}) as  
\begin{equation}
u\left( t\right) :=\underset{I}{\int}\text{ }|\partial _{t}T |^{2}\text{ }ds
=-\frac{d}{dt}\mathcal{F}_{L,\triangle p}\left[T \right] 
\text{,}
\label{eq:energy_ID-3}
\end{equation}
which implies $u \in L^{1}\left( [0,\infty )\right)$. 
From differentiating (\ref{eq:energy_ID-3}) and applying 
(\ref{eq:flow-T}), (\ref{eq:uniform_bdd_nabla^m}), 
we obtain  
\begin{equation*}
\left| u^{\prime }\left( t\right) \right| 
\leq C\left( \left\| \kappa_0\right\|_{L^{2}}, L, |\triangle p| \right) 
\text{.}
\end{equation*}
Therefore, $u\left( t\right) \rightarrow 0$ as $t\rightarrow \infty$. 
In other words, $T_{\infty }(\cdot)$ is independent of $t$ 
and thus, by (\ref{eq:flow-T}), is an equilibrium configuration.
\end{proof}




{\bf Acknowledgement.} 
During working on this project, C.-C. L. would like to acknowledge the support from 
the National Science Council of Taiwan NSC 101-2115-M-003-002 and 
the National Center for Theoretical Sciences in Taipei, Taiwan.


\begin{thebibliography}{99}


\bibitem{Adams03} R. A. Adams and J. J. F. A. Fournier, 
\emph{Sobolev spaces}, 
Second edition. Pure and Applied Mathematics (Amsterdam), 140. 
Elsevier/Academic Press, Amsterdam, 2003. 

\bibitem{Ang91}  S. Angenent, 
\emph{On the formation of singularities in the curve shortening flow}, 
J. Differential Geom. 33 (1991), no. 3, 601-633. 

\bibitem{Antman05} S. S. Antman, 
\emph{Nonlinear problems of elasticity}, 
Second edition. Applied Mathematical Sciences, 107. Springer, New York, 2005. 

\bibitem{BG86} R. Bryant, Robert and P. Griffiths, 
\emph{Reduction for constrained variational problems and $\int\frac{1}{2}\kappa^{2}\text{ }ds$}, 
Amer. J. Math. 108 (1986), no. 3, 525-570. 

\bibitem{BW98} G. Brunnett and J. Wendt, 
\emph{Elastic splines with tension control}, Mathematical methods for curves and surfaces, II (Lillehammer, 1997), 
33-40, Innov. Appl. Math., Vanderbilt Univ. Press, Nashville, TN, 1998. 

\bibitem{Cannon84}  J. R. Cannon, 
\emph{The one-dimensional heat equation},  
With a foreword by Felix E. Browder. Encyclopedia of Mathematics and its Applications, 23. 
Addison-Wesley Publishing Company, Advanced Book Program, Reading, MA, 1984. 

\bibitem{DP13} A. Dall'Acqua, P. Pozzi, 
\emph{A Willmore-Helfrich $L^2$-flow of curves with natural boundary conditions}, 
Preprint ArXiv:1211.0949. 

\bibitem{DLP13} A. Dall'Acqua, C.-C. Lin, P. Pozzi,
\emph{Evolution of open elastic curves in $\mathbb{R}^n$ subject to fixed length and natural boundary conditions}, 
Submitted 2013. 

\bibitem{DKS02}  G. Dziuk, E. Kuwert, and R. Sch\"{a}tzle, 
\emph{Evolution of elastic curves in $\mathbb{R}^{n}$, existence and computation}, 
SIAM J. Math. Anal. 33 (2002), no. 5, 1228-1245.

\bibitem{Evans92}  L. C. Evans, and R. F. Gariepy, 
\emph{Measure theory and fine properties of functions}, 
Studies in Advanced Mathematics, CRC Press, Boca Raton, FL, 1992. 

\bibitem{GH86}  M. Gage and R. S. Hamilton, 
\emph{The heat equation shrinking convex plane curves}, 
J. Differential Geom. 23 (1986), no. 1, 69-96. 

\bibitem{GJ82} M. Golomb and J. Jerome, 
\emph{Equilibria of the curvature functional and manifolds of nonlinear interpolating spline curves}, 
SIAM Journal of Mathematical Analysis, 13 (1982), 421-458. 

\bibitem{Jurdjevic97} V. Jurdjevic, \emph{Geometric control theory}, 
Cambridge Studies in Advanced Mathematics, 52. Cambridge University Press, Cambridge, 1997. 

\bibitem{Koiso96} N. Koiso, \emph{On the motion of a curve towards elastica}, Actes de la Table Ronde de G\'{e}om\'{e}trie Diff\'{e}rentielle (Luminy, 1992), 403-436, S\'{e}min. Congr., 1, Soc. Math. France, Paris, 1996. 

\bibitem{LangSing85} J. Langer and D. A. Singer, 
\emph{Curve straightening and a minimax argument for closed elastic curves}, 
Topology 24 (1985), no. 1, 75-88. 

\bibitem{LangSing87} J. Langer and D. A. Singer, 
\emph{Curve-straightening in Riemannian manifolds}, 
Ann. Global Anal. Geom. 5 (1987), no. 2, 133-150. 

\bibitem{LangSing96} J. Langer and D. A. Singer, 
\emph{Lagrangian aspects of the Kirchhoff elastic rod}, 
SIAM Rev. 38 (1996), no. 4, 605-618. 

\bibitem{Lin12}  C.-C. Lin, 
\emph{$L^2$-flow of elastic curves with clamped boundary conditions}, 
J. Differential Equations 252 (2012), no. 12, 6414-6428.

\bibitem{Linner89}  A. Linn\'{e}r, 
\emph{Some properties of the curve straightening flow in the plane}, 
Trans. Amer. Math. Soc. 314 (1989), no. 2, 605-618. 

\bibitem{Mumford94} D. Mumford, 
\emph{Elastica and computer vision}, Algebraic geometry and its applications 
(West Lafayette, IN, 1990), 491-506, Springer, New York, 1994. 

\bibitem{NO14} M. Novaga and Matteo; S. Okabe, 
\emph{Curve shortening–straightening flow for non-closed planar curves with infinite length}, 
J. Differential Equations 256 (2014), no. 3, 1093-1132.

\bibitem{OS10} D. Oelz and C. Schmeiser, 
\emph{Derivation of a model for symmetric lamellipodia with instantaneous cross-link turnover}, 
Arch. Ration. Mech. Anal. 198 (2010), 
no. 3, 963-980. 


\bibitem{Polden96}  A. Polden, 
\emph{Curves and surfaces of least total curvature and fourth-order flows}, 
Ph.D. dissertation, Universit\"{a}t T\"{u}bingen, T\"{u}bingen, Germany, 1996. 

\bibitem{OSC04} W. K., D. Swigon and B. D. Coleman, 
\emph{Implications of the dependence of the elastic properties of DNA on nucleotide sequence}, 
Philos. Trans. R. Soc. Lond. Ser. A Math. Phys. Eng. Sci. 362 (2004), no. 1820, 1403-1422. 

\bibitem{SH07} E. L. Starostin and G. H. M. van der Heijden, 
\emph{The shape of a M\"{o}bius strip}, 
Nature Materials 6, 563-567 (2007) . 

\bibitem{Wen93}  Y. Wen, 
\emph{$L^2$ flow of curve straightening in the plane}, 
Duke Math. J. 70 (1993), no. 3, 683-698. 

\bibitem{Wen95}  Y. Wen, 
\emph{Curve straightening flow deforms closed plane curves with nonzero rotation number to circles}, 
J. Differential Equations
120 (1995), no. 1, 89-107. 

\bibitem{Willmore00} T. Willmore, 
\emph{Curves}, (English summary) Handbook of differential geometry, Vol. I, 997-1023, 
North-Holland, Amsterdam, 2000. 

\end{thebibliography}
\end{document}